\documentclass[12pt]{article}
\usepackage[a4paper, total={5.9in, 8in}]{geometry}
\usepackage[utf8]{inputenc}
\usepackage{titlesec}
\titleformat{\section}
  {\normalfont\fontsize{14}{15}\bfseries}{\thesection}{1em}{}
\DeclareFontFamily{U}{min}{}

\DeclareFontShape{U}{min}{m}{n}{<-> udmj30}{}

\usepackage{mathtools}
\usepackage[hyphens]{url}
\usepackage{hyperref}
\hypersetup{
    breaklinks=true,
    colorlinks=true,
    linkcolor=blue,
    filecolor=blue,      
    urlcolor=blue,
    citecolor=blue,
    pdftitle={copresheaves},
    }
\usepackage[hyphenbreaks]{breakurl}
\usepackage[bottom]{footmisc}

\usepackage{amsmath}
\usepackage{extpfeil}
\usepackage{amsthm}
\usepackage{amssymb}
\usepackage{booktabs}
\usepackage{adjustbox}
\usepackage{ltablex,array}
\usepackage{longtable}
\usepackage{makecell}
\usepackage{marvosym}
\newcommand*{\faktor}[2]{%
  \raisebox{0.5\height}{\ensuremath{#1}}%
  \mkern-5mu\diagup\mkern-4mu%
  \raisebox{-0.5\height}{\ensuremath{#2}}%
} 
\newcommand\restr[2]{{%
\left.
\kern-
\nulldelimiterspace %
#1 %
\right|_{#2} %
}}
\usepackage{tikz}
\usepackage{tikz-cd}
\tikzcdset{scale cd/.style={every label/.append style={scale=#1},
    cells={nodes={scale=#1}}}}
\usepackage{quiver}
\usetikzlibrary{positioning} \usetikzlibrary{decorations.pathreplacing}
\usetikzlibrary{arrows}
\tikzset{>=stealth}

\theoremstyle{definition}
\newtheorem{theorem}{Theorem}[section]
\newtheorem{proposition}[theorem]{Proposition}

\newtheorem{corollary}[theorem]{Corollary}
\newtheorem{question}[theorem]{Question}
\newtheorem{remark}[theorem]{Remark}
\newtheorem{definition}[theorem]{Definition}

\newtheorem{example}[theorem]{Example}
\theoremstyle{definition}
\usepackage[
backend=biber,
style=numeric,
sorting=nyt
]{biblatex}

\begin{document}
\title{$Sh(B)$-valued models of $(\kappa ,\kappa )$-coherent categories}
\author{Kristóf Kanalas\footnote{Department of Mathematics, Masaryk University, Kotlářská 2 (611 37) Brno, Czech Republic. E-mail: kanalas@mail.muni.cz}}
\date{ }

\maketitle

\begin{abstract}
    A basic technique in model theory is to name the elements of a model by introducing new constant symbols. We describe the analogous construction in the language of syntactic categories/ sites.
    
    As an application we identify $\mathbf{Set}$-valued regular functors on the syntactic category with a certain class of topos-valued models (we will refer to them as "$Sh(B)$-valued models"). For the coherent fragment $L_{\omega \omega }^g \subseteq L_{\omega \omega }$ this was proved by Jacob Lurie, our discussion gives a new proof, together with a generalization to $L_{\kappa \kappa }^g$ when $\kappa $ is weakly compact.
    
    We present some further applications: first, a $Sh(B)$-valued completeness theorem for $L_{\kappa \kappa }^g$ ($\kappa $ is weakly compact), second, that $\mathcal{C}\to \mathbf{Set} $ regular functors (on coherent categories with disjoint coproducts) admit an elementary map to a product of coherent functors.
\vspace{2mm}
$ $

\end{abstract}

\tableofcontents

\section{Introduction}

Categorical logic is the algebraization of logic using categories. Given a theory in a certain fragment of first-order logic one can construct an abstract "category of formulas", called the syntactic category. Formula constructors translate to limits or colimits, while syntactic rules become rules regarding the interaction of colimits with limits (such rules are often called exactness properties).

The typical examples are lex, regular, coherent and Boolean coherent categories which can be identified with Horn, regular, coherent and full first-order theories, see \cite{makkai} for further details.

When moving from fragments of $L_{\omega \omega }$ to fragments of $L_{\lambda \kappa }$, we get infinitary exactness notions ($\kappa $-lex, $\kappa $-regular, $(\lambda ,\kappa )$-coherent). The next section provides a brief overview on this topic, following \cite{presheaftype}.

Then we will study the so-called "method of diagrams" from a categorical point of view. The idea is, that just like a natural transformation out of a representable $\mathcal{C}(x,-)$ is uniquely determined by the image of $id_x$, a lex extension of a $\mathcal{C}\to \mathcal{D}$ lex functor along $\mathcal{C}\to \faktor{\mathcal{C}}{x}$ is uniquely determined by the image of $\Delta _x :x\to x\times x$. This observation (say, "Yoneda for slices") means that moving from $\mathcal{C}$ to $\faktor{\mathcal{C}}{x}$ introduces a generic global element $1\to x$. We can introduce several new global elements by forming a filtered colimit of slices. The properties of these fresh constants will depend on the indexing diagram. When that is $(\int F)^{op}$ for a lex functor $F:\mathcal{C}\to \mathbf{Set}$ (the opposite of the category of elements), the resulting category $\mathcal{C}_F$ has constant names  corresponding to the elements of $F$. We remark that a similar construction has already appeared in \cite{reprGT}, the differences will be discussed at the beginning of Section 3.

In categorical terms, we will prove that given a $\kappa $-lex category $\mathcal{C}$, the functor category $\mathbf{Lex}_{\kappa }(\mathcal{C},\mathbf{Set})$ is a coreflective subcategory in $\mathcal{C}\downarrow \mathbf{Lex}_{\kappa }^{\sim }$, the (2,1)-category of $\kappa $-lex functors out of $\mathcal{C}$ (see Theorem \ref{addingconst}, and Corollary \ref{adjwithsites} for a version with $\kappa $-sites). 

In Section 4.~this method will be applied to give a new proof of Lurie's theorem (\cite[Theorem 11., Lecture 16X]{lurie}), which says that $\mathcal{C}\to \mathbf{Set}$ regular functors on a small pretopos $\mathcal{C}$ can be identified with $\mathcal{C}\to Sh(X)$ coherent functors, where $X$ is a non-fixed Stone-space. 

We will obtain the following generalization: when $\kappa $ is weakly compact and $\mathcal{C}$ is $(\kappa ,\kappa )$-coherent with $\kappa $-small disjoint coproducts, $\mathcal{C}\to \mathbf{Set}$ $\kappa $-regular functors can be identified with $\mathcal{C}\to Sh(B)$ $(\kappa , \kappa )$-coherent functors (Theorem \ref{Lurie}). Here $B$ is a non-fixed $(\kappa ,\kappa )$-coherent Boolean-algebra, considered with the topology $\tau _{\kappa -coh}$, formed by $\kappa $-small unions. 

In what follows, we will refer to such functors as $Sh(B)$-valued models. Let us briefly comment on the terminology: in the finitary case $Sh(B)$ can mean two different things. If $B$ is a complete Boolean-algebra and the topology consists of all unions then $Sh(B)=Sh(B,\tau _{can})$-valued models are the same as Boolean-valued models, an important concept which appears e.g.~in forcing and in completeness theorems for $L_{\infty ,\omega }$ (cf.~\cite[Section 4]{makkai}). When $B$ is an arbitrary Boolean-algebra considered with the finite union topology, then $Sh(B)=Sh(B,\tau _{coh})$ is the same as the spatial topos $Sh(X)$ where $X=Spec(B)$ is the Stone-space of $B$. Hence Lurie calls $Sh(B,\tau_{coh})$-valued models "parametrized", see \cite[Lecture 16X]{lurie}. Note that $Sh(B,\tau _{coh})=Sh(Id(B),\tau _{can})$, where $Id(B)$ is the ideal-completion of $B$. This is a consequence of the comparison lemma, applied to the full subcategory $Sub(1)=Id(B)\subseteq Sh(B,\tau _{coh}) $. So parametrized models can be seen as Heyting-valued models, where the complete Heyting-algebra is of the form $Id(B)$. 

We use the term "$Sh(B)$-valued", instead of "parametrized", as there is no Stone-duality for $(\kappa ,\kappa )$-coherent Boolean-algebras ($\kappa $ is weakly compact), so we can not conclude that $Sh(B)=Sh(B,\tau _{\kappa -coh})$ is spatial, i.e.~$Sh(B)$-valued models may not be parametrized by a topological space.

In any case, the above mentioned result is promising, as it says that something model-theoretically interesting ($Sh(B)$-valued models, a special class of Heyting-valued models) is the same as something category-theoretically simple ($\mathbf{Set}$-valued regular functors). Indeed, we have some applications:

First, a $Sh(B)$-valued completeness theorem for $(\kappa ,\kappa )$-coherent logic will easily follow (Theorem \ref{completeness}). Then we will compare $\mathcal{C}\to \mathbf{Set}$ regular functors with products of $\mathcal{C}\to \mathbf{Set}$ coherent functors. A natural transformation between lex functors is said to be elementary if the naturality squares at monomorphisms are pullbacks. We will prove that if $\mathcal{C}$ is a coherent category with disjoint coproducts, then every regular functor $F:\mathcal{C}\to \mathbf{Set}$ admits an elementary map to a product of coherent functors (the analogous statement for $(\kappa ,\kappa )$-coherent categories is true when $\kappa $ is strongly compact, see Theorem \ref{regtoprodcoh}). Along the way we will make some observations regarding elementary maps, e.g.~that if the unit of a geometric morphism is pointwise mono then it is elementary (Corollary \ref{unitelem}).

Finally, we mention a third application. In the book \cite{rosicky} the following was asked: "Is it true that in $\lambda $-accessible categories $\lambda $-pure maps are regular monomorphisms?" In \cite{purestrictmono} a partial positive answer is obtained, and Theorem \ref{regtoprodcoh} is the key ingredient of that proof.

\section{An overview of infinitary exactness properties}

For the convenience of the reader we provide a brief overview of the key ideas regarding compatibility between colimits and $\kappa $-small limits. For a more detailed discussion we refer to \cite{presheaftype}.

\begin{definition}
  ($\kappa $ is an infinite regular cardinal.)  A category is $\kappa $-lex if it has $\kappa $-small limits. A functor is $\kappa $-lex if it preserves $\kappa $-small limits.
\end{definition}

\begin{definition}
   ($\kappa $ is an infinite regular cardinal.) A category is $\kappa $-regular if
    \begin{itemize}
        \item[i)] it is $\kappa $-lex,
        \item[ii)] has pullback-stable effective epi-mono factorization,
        \item[iii)] and the limit of a shorter than $\kappa $, continuous (i.e.~at limit steps we have limits), co-well-ordered chain of effective epis 
\[
(x_0\twoheadleftarrow x_1\twoheadleftarrow \dots x_{\alpha }\twoheadleftarrow \dots )_{\alpha <\gamma <\kappa }
\]
is an effective epi. (Such limits are called transfinite cocompositions.)
    \end{itemize}

    A functor is $\kappa $-regular if it preserves $\kappa $-small limits and effective epimorphisms.
\end{definition}

$\kappa $-regular categories were defined by Michael Makkai in \cite{barrex}, as the correct notion of compatibility between effective epimorphisms and $\kappa $-small limits. 

Compatibility between effective epimorphic families and $\kappa $-small limits is a much more recent topic. It has its origin in the works of Christian Espíndola (see \cite{inflog}). The idea is very natural: when we glue together some epimorphic families (resulting in a cotree), the transfinite cocompositions of the branches should form an epimorphic family on the root. We will work with the following definition:

\begin{definition}
\label{lkcoherent}
    ($\lambda \geq \kappa $ are infinite regular cardinals.) A category is $(\lambda ,\kappa )$-coherent if
    \begin{itemize}
        \item[i)] it is $\kappa $-lex,
        \item[ii)] has pullback-stable effective epi-mono factorization,
        \item[iii)] has pullback-stable $\lambda $-small unions,
        \item[iv)] and given a rooted cotree (that is: the opposite of a rooted tree)

\[
\adjustbox{scale=0.75,center}{        

\begin{tikzcd}
	& {} & {} \\
	& {} & {} &&& \dots \\
	x & {x_{i_0}} & {x_{i_0,i_1}} & \dots & {x_{i_0,\dots }} & {x_{i_0,\dots i_{\omega }}} & \dots \\
	& \dots & {}
	\arrow[from=1-2, to=2-2]
	\arrow[from=1-3, to=2-2]
	\arrow[from=2-2, to=3-1]
	\arrow[from=2-3, to=2-2]
	\arrow[from=2-6, to=3-5]
	\arrow[from=3-2, to=3-1]
	\arrow[from=3-3, to=3-2]
	\arrow[from=3-4, to=3-3]
	\arrow[from=3-6, to=3-5]
	\arrow[from=3-7, to=3-6]
	\arrow[from=4-2, to=3-1]
	\arrow[from=4-3, to=3-2]
\end{tikzcd}
}
\]

such that on every vertex its predecessors form an extremal epimorphic family of size $<\lambda$, and every branch is $<\kappa $, continuous, it follows that the transfinite cocompositions of the branches form an extremal epimorphic family on the root.
    \end{itemize}

    A functor is $(\lambda ,\kappa )$-coherent if it preserves $\kappa $-small limits, effective epimorphisms and $\lambda $-small unions.
\end{definition}

Here we shall make an important remark. First recall, that an infinite cardinal $\kappa $ has the tree property, if given a tree of height $\kappa $ such that every level is of size $<\kappa $, it follows that the tree has a branch of size $\kappa $ (cf.~\cite[p.~120]{jech}). 

In our case, since $\lambda $ is regular there are $<\lambda $ objects in each level below a cardinal $\mu $, assuming $\alpha _i <\lambda $ for $i<\gamma <\mu $ implies $\prod _i \alpha _i <\lambda $. By the regularity of $\lambda $ that is the same as $\alpha <\lambda ,\gamma <\mu \Rightarrow \alpha ^{\gamma }<\lambda $, and we will abbreviate this by $(<\lambda )^{<\mu }=(<\lambda )$. Since every branch is $<\kappa $, the tree has height $\leq \kappa $. Therefore if $\lambda >\kappa $ with $(<\lambda )^{\kappa }=(<\lambda )$ or if $\lambda =\kappa $ has the tree property and satisfies $(<\kappa )^{<\kappa }=(<\kappa )$, it follows that the tree has $<\lambda $-many branches. The latter case is equivalent to $\kappa $ being weakly compact.

This will be important for two reasons:

First, the careful reader may have noticed a confusing imbalance in the above definition: it speaks about effective epimorphisms and extremal epimorphic families. It is well-known that (using the presence of $<\lambda $ unions) a $<\lambda $ extremal epimorphic family is also effective epimorphic, so this is only an issue when the resulting cotrees can have $\geq \lambda $-many branches. In this case we choose to work with the weaker notion (extremal epic) mainly as it is sufficient for proving a completeness theorem (see \cite{presheaftype}).

The other reason is that we shall often form the $\kappa $-filtered colimit of $(\kappa ,\kappa )$-coherent categories and without assuming $\kappa $ to be weakly compact, we can not conclude that the resulting category is also $(\kappa ,\kappa )$-coherent (as we would have to check a family of size $\geq \kappa $ which is not coming from a single object (category) of the indexing diagram).

We proceed with the definition of a $\kappa $-site:

\begin{definition}
    ($\kappa $ is an infinite regular cardinal.) A $\kappa $-site is a pair $(\mathcal{C},E)$ where $\mathcal{C}$ is a $\kappa $-lex category and $E$ is any collection of families (i.e.~diagrams of the form $(u_i\to x)_i$), such that $\{ 1\to 1 \}$ belongs to $E$.
\end{definition}

Some typical examples are: $\kappa $-lex categories with $E=\{1\to 1 \} $, $\kappa $-regular categories with  $E=\{ \text{effective epimorphisms}\}$ and $(\lambda ,\kappa )$-coherent categories with $E=\{<\lambda \text{ effective epimorphic families} \}$.

The idea is that $\kappa $-sites can encode the same amount of information as $(\lambda ,\kappa )$-coherent categories, but they don't have to satisfy any compatibility conditions.

\begin{definition}
    A Grothendieck-topology (on a $\kappa $-lex category $\mathcal{C}$) is said to be a $\kappa $-topology, if it is closed under building cotrees as in Definition \ref{lkcoherent} iv) (that is, assuming that on each vertex the predecessors form a cover, it follows that the transfinite cocompositions of the branches form a cover on the root).
\end{definition}

\begin{example}
\label{ktopolsubcanon}
    If $\mathcal{C}$ is $(\kappa ,\kappa )$-coherent and $\kappa $ is weakly compact, then the set of $\kappa $-small effective epimorphic families is a $\kappa $-topology on $\mathcal{C}$.
\end{example}

 Given any $\kappa $-site $(\mathcal{C},E)$ one can generate a $\kappa $-topology $\langle E\rangle _{\kappa }$ out of $E$, by first closing it under pullbacks, then under the formation of cotrees as in Definition \ref{lkcoherent} iv). This leads to the notion of maps between $\kappa $-sites:

\begin{definition}
    A morphism $F:(\mathcal{C},E)\to (\mathcal{D},E')$ between $\kappa $-sites is a $\kappa $-lex functor $F:\mathcal{C}\to \mathcal{D}$ which takes $E$-families (equivalently: $\langle E\rangle _{\kappa }$-families) to the ones in $\langle E'\rangle _{\kappa }$.
\end{definition}

By taking sheaves on $\kappa $-sites we get $\kappa $-toposes:

\begin{definition}
    ($\kappa $ is an infinite regular cardinal.) A Grothendieck-topos is said to be a $\kappa $-topos if it is an $(\infty ,\kappa )$-coherent category. A morphism of $\kappa $-toposes is a geometric morphism whose inverse image is $\kappa $-lex.
\end{definition}

We recall the following results from \cite{presheaftype}, mainly to fix the notation:

\begin{theorem}
    If $(\mathcal{C},E)$ is a $\kappa $-site, then $Sh(\mathcal{C},\langle E\rangle _{\kappa })$ is a $\kappa $-topos.

    If $F:(\mathcal{C},E)\to (\mathcal{D},E')$ is a morphism of $\kappa $-sites, then the induced geometric morphism is a morphism of $\kappa $-toposes. We will write $F_*$ for the direct image and $F^*$ for the inverse image map.

    $Y:\mathcal{C}\to \mathbf{Set}^{\mathcal{C}^{op}}$ will denote the Yoneda-embedding and $\# :\mathbf{Set}^{\mathcal{C}^{op}}\to Sh(\mathcal{C},\langle E\rangle _{\kappa })$ will stand for the sheafification map. It follows that $\# $ is $\kappa $-lex, moreover also $\# Y:\mathcal{C}\to Sh(\mathcal{C},\langle E\rangle _{\kappa })$ is $\kappa $-lex, and it turns $E$-families to (effective) epimorphic ones. In what follows we will refer to such functors either as "models" or as "$\kappa $-lex $E$-preserving". We will often write $\widehat{x}$ instead of $\#Y(x)$.

    Given a $\kappa $-topos $\mathcal{E}$, precomposition with $\# Y$ yields an equivalence between $Sh(\mathcal{C},\langle E\rangle _{\kappa })\to \mathcal{E}$ $\kappa $-lex cocontinuous functors (equivalently~$\mathcal{E}\to Sh(\mathcal{C})$ $\kappa $-topos maps), and $\mathcal{C}\to \mathcal{E}$ $\kappa $-lex $E$-preserving functors. 

    To avoid the overuse of $*$, we write $(\#Y)^{\circ }$ for this precomposition functor. More generally, in any context $f^{\circ }$ will denote precomposition with $f$ and $f_{\circ}$ will stand for postcomposition. We shall write $f^{-1}$ for the map "pulling back along $f$", and when $f$ happens to be invertible, its inverse will be called $\overline{f}$.
\end{theorem}

\section{The "method of diagrams" categorically}

In model theory the "method of diagrams" means the following: starting with a model $M$ of a theory $T\subseteq L_{\omega \omega}$, one first extends the signature by adding constant symbols corresponding to the elements of $M$ (which yields $L_M\supseteq L$), then extends $T$ by adding all closed $L_M$-formulas which are valid in $M$. The resulting theory $Diag(M)\subseteq (L_M)_{\omega \omega }$ is the diagram of $M$. Of course, one could also do this inside the coherent (or any other) fragment, that is, to start with $T\subseteq L_{\omega \omega}^g$ and add only coherent formulas with parameters from $M$, to obtain the positive diagram $Diag^+(M)\subseteq (L_M)_{\omega \omega }^g$.

As a result, a model of $Diag(M)$ is the same as a model of $T$ (say $N$), together with an elementary embedding $M\to N$, obtained from the interpretation of the new constant symbols. Similarly, a model of $Diag^+(M)$ is the same as a model $N$ of $T$, together with a homomorphism $M\to N$. 

The "main facts" of categorical logic say that coherent theories, models and homomorphisms are the same as coherent categories, ($\mathbf{Set}$-valued) coherent functors and natural transformations (cf.~\cite{makkai}). So by going back-and-forth, one gets that given a coherent functor $M:\mathcal{C}\to \mathbf{Set}$, there exists a coherent category $\mathcal{C}_M$, such that the category of $\mathcal{C}_M\to \mathbf{Set}$ coherent functors is equivalent to the coslice category $M\downarrow \mathbf{Coh}(\mathcal{C},\mathbf{Set})$.

Our philosophy is that theorems in categorical logic should have a model-theoretic intuition and a purely category-theoretic proof. First, because such proofs are cleaner, especially in the infinite-quantifier setup. Second, because they can lead to natural generalizations, which are not visible from the logic side. 

The above mentioned theorem (formulated with lex categories and sites, rather than coherent categories) is proved as part of \cite[Proposition 2.1.]{reprGT}. We prove it (for $\kappa $-lex categories and $\kappa $-sites) as Corollary \ref{Dlocsmall}. Their proof is based on the same idea as ours (we called it "Yoneda for slices", see the introduction), but there are some essential differences. 

Mainly, our colimit (of slices) is a 2-categorical filtered colimit, the colimit in \cite{reprGT} is a 1-categorical sequential colimit. This way we can escape a strictification process (forcing pullbacks to be determined up to equality), as well as a transfinite construction, involving an enumeration of the elements in the image of the given functor $M$.

More importantly, the fact that our colimit has a canonical indexing diagram $(\int M)^{op}$, makes $M\mapsto \mathcal{C}_M$ functorial, which in particular allows us to prove that $\mathbf{Lex}_{\kappa }(\mathcal{C},\mathbf{Set})$ is a full coreflective subcategory in $\mathcal{C}\downarrow \mathbf{Lex}_{\kappa }^{\sim }$, the (2,1)-category   of $\kappa $-lex functors out of $\mathcal{C}$ with small codomain, see Theorem \ref{addingconst}. It also leads to an explicit description of $\mathcal{C}_M$ which will often be useful. 

\begin{definition}
    ($\kappa $ is an infinite regular cardinal.) Write $\mathbf{Lex}^{\sim }_{\kappa }$ for the (2,1)-category of small $\kappa $-lex categories, $\kappa $-lex functors and natural isomorphisms. 
    
    Given a (small) $\kappa $-lex category $\mathcal{C}$, we write $\mathcal{C}\downarrow \mathbf{Lex}_{\kappa }^{\sim }$ for the corresponding coslice: its objects are $\mathcal{C}\to \mathcal{D}$ $\kappa $-lex functors, a 1-cell from $M:\mathcal{C}\to \mathcal{D}$ to $N:\mathcal{C}\to \mathcal{D}'$ is a $\kappa $-lex functor $F:\mathcal{D}\to \mathcal{D}'$ together with a natural isomorphism $\nu _F: FM\Rightarrow N$, and a 2-cell from $(F,\nu _F)$ to $(G,\nu _G)$ is a natural isomorphism $\theta :F\Rightarrow G$, making

\[\begin{tikzcd}
	&&& {\mathcal{D}} \\
	{\mathcal{C}} \\
	&&& {\mathcal{D}'}
	\arrow[""{name=0, anchor=center, inner sep=0}, "F"{description}, curve={height=12pt}, from=1-4, to=3-4]
	\arrow[""{name=1, anchor=center, inner sep=0}, "G"{description}, curve={height=-12pt}, from=1-4, to=3-4]
	\arrow[""{name=2, anchor=center, inner sep=0}, "M", from=2-1, to=1-4]
	\arrow[""{name=3, anchor=center, inner sep=0}, "N"', from=2-1, to=3-4]
	\arrow["\theta", shorten <=5pt, shorten >=5pt, Rightarrow, from=0, to=1]
	\arrow["{\nu _F}"', curve={height=6pt}, shorten <=5pt, shorten >=5pt, Rightarrow, from=2, to=3]
	\arrow["{\nu _G}", curve={height=-6pt}, shorten <=5pt, shorten >=5pt, Rightarrow, from=2, to=3]
\end{tikzcd}\]    
    commute (i.e.~$\nu _G\circ (\theta M) =\nu _F$).
\end{definition}

\begin{theorem}
\label{delta}
     Take $\mathcal{C}$, $\mathcal{D}$, and $M:\mathcal{C}\to \mathcal{D}$ in $\mathbf{Lex}^{\sim }_{\kappa } $, and choose $x\in \mathcal{C}$. Then we have an equivalence of groupoids:
     
    \[
    \mathcal{C}\downarrow \mathbf{Lex}^{\sim }_{\kappa } \left(
 \mathcal{C}\xrightarrow{!^{-1}} \faktor{\mathcal{C}}{x}, \ \mathcal{C}\xrightarrow{M}\mathcal{D}\right) \ \ \xrightarrow{\delta } \ \ \mathcal{D}(M(1),M(x))
    \]
Here $!:x\to 1$ is the unique map to the terminal object and $!^{-1}$ is the functor "pulling back along $!$", so it sends $f:y\to z$ to $f\times id_x: y\times x\to z\times x$. The product objects are chosen arbitrarily, except that for notational convenience we assume $1\times x=x$. The set $\mathcal{D}(M(1),M(x))$ is seen as a discrete groupoid. The map $\delta $ takes a 1-cell 

\[\begin{tikzcd}
	&& {\faktor{\mathcal{C}}{x}} \\
	{\mathcal{C}} \\
	&& {\mathcal{D}}
	\arrow[""{name=0, anchor=center, inner sep=0}, "{!^{-1}}", from=2-1, to=1-3]
	\arrow[""{name=1, anchor=center, inner sep=0}, "M"', from=2-1, to=3-3]
	\arrow["F", from=1-3, to=3-3]
	\arrow["{\nu _F}", "{\cong }"', shift left=4, shorten <=4pt, shorten >=4pt, Rightarrow, from=0, to=1]
 \tag{$*$}
\end{tikzcd}\]
to $\nu _{F,x} \circ F(\Delta _x) \circ \overline{\nu _{F,1}}:M(1)\to M(x)$ with $\Delta _x$ being the map
\[\begin{tikzcd}
	x && {x\times x} \\
	& x
	\arrow["{\Delta _x = \langle id_x ,id_x \rangle }", from=1-1, to=1-3]
	\arrow["{id_x}"', from=1-1, to=2-2]
	\arrow["{\pi _2}", from=1-3, to=2-2]
\end{tikzcd}\]
    
\end{theorem}

\begin{proof}
    This is a functor, i.e.~given a 1-cell $\alpha :F\Rightarrow G$ in the Hom-groupoid we have $\delta (F,\nu _F) = \delta (G,\nu _G)$, by the commutativity of

\[\begin{tikzcd}
	{F(id_x)} && {F(\pi _2)} \\
	& {G(id_x)} && {G(\pi _2)} \\
	{M(1)} && {M(x)}
	\arrow["{F(\Delta _x)}", from=1-1, to=1-3]
	\arrow["{G(\Delta _x)}"{pos=0.2}, from=2-2, to=2-4]
	\arrow["{\alpha }"', from=1-1, to=2-2]
	\arrow["{\alpha }", from=1-3, to=2-4]
	\arrow["{\nu _{F,1}}"', from=1-1, to=3-1]
	\arrow["{\nu _{G,1}}", from=2-2, to=3-1]
	\arrow["{\nu _{F,x}}"'{pos=0.8}, from=1-3, to=3-3]
	\arrow["{\nu_{G,x}}", from=2-4, to=3-3]
\end{tikzcd}\]

We define the quasi-inverse $F_{-}$ of $\delta $ as follows. Given $a:1=M(1)\to M(x)$ let $F_a$ be the following functor: it takes a map 
\[\begin{tikzcd}
	y && z \\
	& x
	\arrow["h"', from=1-1, to=2-2]
	\arrow["k", from=1-3, to=2-2]
	\arrow["f", from=1-1, to=1-3]
\end{tikzcd}\]
to the action of its $M$-image on the fiber over $a$, i.e.~to

\[
\adjustbox{scale=0.9}{
\begin{tikzcd}
	{F_{a}(h)} && {M(y)} \\
	& pb \\
	{F_{a}(k)} && {M(z)} \\
	& pb \\
	1 && {M(x)}
	\arrow["{M(h)}"{description, pos=0.3}, curve={height=-30pt}, from=1-3, to=5-3]
	\arrow["a"', from=5-1, to=5-3]
	\arrow[curve={height=30pt}, from=1-1, to=5-1]
	\arrow[from=1-1, to=1-3]
	\arrow["{M(k)}"{description}, from=3-3, to=5-3]
	\arrow["{M(f)}"{description}, from=1-3, to=3-3]
	\arrow[from=3-1, to=5-1]
	\arrow[from=3-1, to=3-3]
	\arrow["{F_{a}(f)}"{description}, dashed, from=1-1, to=3-1]
\end{tikzcd}
}
\]

Of course, whenever a functor arises from taking pullbacks, it means that we fix arbitrary choices on the object level, and arrows are sent to the corresponding unique induced maps. This functor $F_a$ is $\kappa $-lex, because products are wide pullbacks over $x$, this is preserved by $M$ and it is pulled back to a wide pullback over $1$, i.e.~to a product. Similarly, the terminal object and equalizers are preserved. 

$F_a\circ !^{-1}$ is isomorphic to $M$: since $M$ preserves products and pullbacks 
we get

\[
\adjustbox{scale=0.9}{
\begin{tikzcd}
	{M(y)=M(y)\times 1} && {M(y\times x)} \\
	& pb \\
	{M(z)=M(z)\times 1} && {M(z\times x)} \\
	& pb \\
	1 && {M(x)}
	\arrow["a", from=5-1, to=5-3]
	\arrow["{M(\pi _2)}", from=3-3, to=5-3]
	\arrow["{M(f\times id_x)}", from=1-3, to=3-3]
	\arrow["{ id_{M(z)} \times a}", from=3-1, to=3-3]
	\arrow[from=3-1, to=5-1]
	\arrow["{id_{M(y)} \times a}", from=1-1, to=1-3]
	\arrow["{M(f)}"', from=1-1, to=3-1]
\end{tikzcd}
}
\]
As we may have chosen a different pullback for $M(y\times x)\xrightarrow{M(\pi _2)} M(x)\xleftarrow{a}1$ when we defined $F_a(!^{-1}y)$, the comparison maps yield an isomorphism in

\[
\adjustbox{scale=0.87}{
\begin{tikzcd}
	&& {\faktor{\mathcal{C}}{x}} \\
	{\mathcal{C}} \\
	&& {\mathcal{D}}
	\arrow[""{name=0, anchor=center, inner sep=0}, "{!^{-1}}", from=2-1, to=1-3]
	\arrow[""{name=1, anchor=center, inner sep=0}, "M"', from=2-1, to=3-3]
	\arrow["{F_a}", from=1-3, to=3-3]
	\arrow["{\nu _{F_a}}", shorten <=4pt, shorten >=4pt, Rightarrow, from=0, to=1]
\end{tikzcd}
}
\]

To see that $\delta \circ F_{-}$ equals the identity, we need that given $a:M(1)\to M(x)$, it equals $M(1)\xrightarrow{\overline{\nu _{{F_a},1}}}F_a(id_x)\xrightarrow{F_a(\Delta _x)}F_a(\pi _2)\xrightarrow{\nu _{{F_a},x}}M(x)$. This follows, as $F_a$ evaluated on $\Delta _x$ is the fiber of $M(\Delta _x)$ over $a$, which is $a$ itself (after identifying the pullback objects via $\nu _{F_a}$):

\[
\adjustbox{scale=0.9}{
\begin{tikzcd}
	{M(1)} && {M(x)} \\
	& pb \\
	{M(x)} && {M(x\times x)} \\
	& pb \\
	{M(1)} && {M(x)}
	\arrow["{M(\pi _2)}", from=3-3, to=5-3]
	\arrow["a"', from=5-1, to=5-3]
	\arrow[from=3-1, to=5-1]
	\arrow["{id_{M(x)}\times a}", from=3-1, to=3-3]
	\arrow["{M(\Delta _x)}", from=1-3, to=3-3]
	\arrow["a"', from=1-1, to=3-1]
	\arrow["a", from=1-1, to=1-3]
\end{tikzcd}
}
\]

Conversely, given an extension $(F,\nu _F)$ as in $(*)$, there exists a unique natural isomorphism $\theta $ from $F$ to $F_{\delta (F)}$ satisfying $\nu _{F_{\delta (F)}} \circ (\theta !^{-1})=\nu _F$. This follows, as
\[\begin{tikzcd}
	y && {y\times x} \\
	x && {x\times x} \\
	& x
	\arrow["{\langle id_y,h\rangle}", from=1-1, to=1-3]
	\arrow["h"', from=1-1, to=2-1]
	\arrow["h"{description, pos=0.7}, from=1-1, to=3-2]
	\arrow["{h\times id_x}", from=1-3, to=2-3]
	\arrow["{\pi _2}"{description, pos=0.7}, from=1-3, to=3-2]
	\arrow["{\Delta _x}", from=2-1, to=2-3]
	\arrow["{id_x}"', from=2-1, to=3-2]
	\arrow["{\pi _2}", from=2-3, to=3-2]
\end{tikzcd}\]
is a pullback in $\faktor{\mathcal{C}}{x}$, which the lex extension $F$ preserves, hence in
\[\begin{tikzcd}
	{F(h)} && {F(!^{-1}y)} \\
	& {F_{\delta (F)}} && {M(y)} \\
	{F(!^{-1}1)} && {F(!^{-1}x)} \\
	& {M(1)} && {M(x)}
	\arrow[from=1-1, to=1-3]
	\arrow["{\theta _h}"', dashed, from=1-1, to=2-2]
	\arrow[from=1-1, to=3-1]
	\arrow["{\nu _{F,y}}", from=1-3, to=2-4]
	\arrow["{F(!^{-1}h)}"'{pos=0.2}, from=1-3, to=3-3]
	\arrow[from=2-2, to=2-4]
	\arrow[from=2-2, to=4-2]
	\arrow["{M(h)}", from=2-4, to=4-4]
	\arrow["{F(\Delta _x)}"{pos=0.2}, from=3-1, to=3-3]
	\arrow["{\nu _{F,1}}"', from=3-1, to=4-2]
	\arrow["{\nu _{F,x}}", from=3-3, to=4-4]
	\arrow["{\delta (F)}"', from=4-2, to=4-4]
\end{tikzcd}\]
both the front and the back face is a pullback, and the induced map $\theta _h$ is an isomorphism.

\end{proof}

\begin{proposition}
    $\delta $ is natural both in $x$ and in $M$.
\end{proposition}

\begin{proof}
    In $x$: given $f:x\to y$ we need the commutativity of

\[\begin{tikzcd}
	{\mathcal{C}\downarrow \mathbf{Lex}^{\sim }_{\kappa } (\mathcal{C}\to \faktor{\mathcal{C}}{x}, \mathcal{C}\xrightarrow{M}\mathcal{D})} && {\mathcal{D}(M(1),M(x))} \\
	& {} \\
	{\mathcal{C}\downarrow \mathbf{Lex}^{\sim }_{\kappa } (\mathcal{C}\to \faktor{\mathcal{C}}{y}, \mathcal{C}\xrightarrow{M}\mathcal{D})} && {\mathcal{D}(M(1),M(y))}
	\arrow["{-\circ f^{-1}}", from=1-1, to=3-1]
	\arrow["{\delta _{x,M}}", from=1-1, to=1-3]
	\arrow["{M(f)\circ -}", from=1-3, to=3-3]
	\arrow["{\delta _{y,M}}", from=3-1, to=3-3]
\end{tikzcd}\]
Let $F:\faktor{\mathcal{C}}{x}\to \mathcal{D}$ be an extension and apply it to the following commutative square in $\faktor{\mathcal{C}}{x}$
\[\begin{tikzcd}
	& {y\times x} && {x\times x} \\
	x &&&& x \\
	&& x
	\arrow[Rightarrow, no head, from=2-1, to=2-5]
	\arrow["{id_x}"{description}, from=2-5, to=3-3]
	\arrow["{id_x}"{description}, from=2-1, to=3-3]
	\arrow["{\langle f,id_x\rangle=f^{-1}(\Delta _y)}", from=2-1, to=1-2]
	\arrow["\Delta _x"', from=2-5, to=1-4]
	\arrow["{f\times id_x}"', from=1-4, to=1-2]
	\arrow["{\pi _2^{y\times x}}"{description, pos=0.3}, from=1-2, to=3-3]
	\arrow["{\pi _2^{x\times x}}"{description, pos=0.3}, from=1-4, to=3-3]
\end{tikzcd}\]
to get the diagram below in which the two composites $M(1)\to M(y)$ are equal

\[\begin{tikzcd}
	&& {M(y)} && {M(x)} \\
	{M(1)} && {F(\pi _2^{y\times x})} && {F(\pi _2^{x\times x})} \\
	{F(id_x)} &&& {F(id_x)}
	\arrow["{M(f)}"', from=1-5, to=1-3]
	\arrow["{\nu _{F,y}}", from=2-3, to=1-3]
	\arrow["{\nu _{F,x}}"', from=2-5, to=1-5]
	\arrow["{F(f\times id_x)}"', from=2-5, to=2-3]
	\arrow["{\nu _{F,1}}", from=3-1, to=2-1]
	\arrow["{F(f^{-1}\Delta _y)}"{description}, from=3-1, to=2-3]
	\arrow[Rightarrow, no head, from=3-1, to=3-4]
	\arrow["{F(\Delta _x)}"', from=3-4, to=2-5]
\end{tikzcd}\]

In $M$: Given a 1-cell $(H,\nu _H):(\mathcal{C}\xrightarrow{M}\mathcal{D})\to (\mathcal{C}\xrightarrow{N}\mathcal{D}')$, we have to show the commutativity of

\[\begin{tikzcd}
	{\mathcal{C}\downarrow \mathbf{Lex}^{\sim }_{\kappa } (\mathcal{C}\to \faktor{\mathcal{C}}{x}, \mathcal{C}\xrightarrow{M}\mathcal{D})} && {\mathcal{D}(M(1),M(x))} \\
	& {} \\
	{\mathcal{C}\downarrow \mathbf{Lex}^{\sim }_{\kappa } (\mathcal{C}\to \faktor{\mathcal{C}}{x}, \mathcal{C}\xrightarrow{N}\mathcal{D}')} && {\mathcal{D}'(N(1),N(x))}
	\arrow["{(H,\nu _H) \circ -}", from=1-1, to=3-1]
	\arrow["{\delta _{x,M}}", from=1-1, to=1-3]
	\arrow["{\nu _{H,x} \circ H(-)\circ \overline{\nu _{H,1}}}", from=1-3, to=3-3]
	\arrow["{\delta _{x,N}}", from=3-1, to=3-3]
\end{tikzcd}\]

This follows, as given an extension $F:\faktor{\mathcal{C}}{x}\to \mathcal{D}$, both routes result in
\[
N(1)\xrightarrow{\overline{\nu _{H,1}}}HM(1)\xrightarrow{\overline{H(\nu _{F,1})}}HF(id_x)\xrightarrow{HF(\Delta _x)} HF(\pi _2)\xrightarrow{H(\nu _{F,x})}HM(x)\xrightarrow{\nu _{H,x}} N(x)
\]
\end{proof}

Extending this observation from slices to their $\kappa $-filtered colimits will yield the construction we promised. 

\begin{theorem}
\label{addingconst}
    There is a (2,1)-adjunction:
\[\begin{tikzcd}
	{\mathbf{Lex}_{\kappa }(\mathcal{C},\mathbf{Set})} &&& {\mathcal{C}\downarrow \mathbf{Lex}^{\sim }_{\kappa }}
	\arrow[""{name=0, anchor=center, inner sep=0}, "{\mathcal{C}_{-}}", curve={height=-12pt}, hook, from=1-1, to=1-4]
	\arrow[""{name=1, anchor=center, inner sep=0}, "\Gamma", curve={height=-12pt}, from=1-4, to=1-1]
	\arrow["\dashv"{anchor=center, rotate=-90}, draw=none, from=0, to=1]
\end{tikzcd}\]
Being a (2,1)-adjunction means that we have an equivalence of groupoids
\[
\mathcal{C}\downarrow \mathbf{Lex}^{\sim }_{\kappa }(\mathcal{C}\xrightarrow{\varphi _K} \mathcal{C}_K, \mathcal{C}\xrightarrow{M}\mathcal{D}) \simeq Nat(K,\Gamma M)
\]
natural both in $K:\mathcal{C}\to \mathbf{Set}$ and in $M:\mathcal{C}\to \mathcal{D}$. (The 2-cells of $\mathbf{Lex}_{\kappa }(\mathcal{C},\mathbf{Set})$ are just the identities, and $Nat(K,\Gamma M)$ is a discrete groupoid.)

$\Gamma \circ \mathcal{C}_{-}$ is naturally isomorphic to the identity functor on $\mathbf{Lex}_{\kappa }(\mathcal{C},\mathbf{Set})$.
\end{theorem}

\begin{proof}
   We start by defining the above functors. $\Gamma $ fixes for each $\mathcal{D}\in \mathbf{Lex}^{\sim }_{\kappa }$ a terminal object $1_{\mathcal{D}}$ and maps $M:\mathcal{C}\to \mathcal{D}$ to $\mathcal{D}(1_{\mathcal{D}},-)\circ M$. A 2-cell is taken to 
\[\begin{tikzcd}
	&&& {\mathcal{D}} \\
	{\mathcal{C}} &&&&&& {\mathbf{Set}} \\
	&&& {\mathcal{D}'}
	\arrow[""{name=0, anchor=center, inner sep=0}, "M", from=2-1, to=1-4]
	\arrow[""{name=1, anchor=center, inner sep=0}, "N"', from=2-1, to=3-4]
	\arrow[""{name=2, anchor=center, inner sep=0}, "F"{description}, curve={height=12pt}, from=1-4, to=3-4]
	\arrow[""{name=3, anchor=center, inner sep=0}, "G"{description}, curve={height=-12pt}, from=1-4, to=3-4]
	\arrow[""{name=4, anchor=center, inner sep=0}, "{\mathcal{D}(1_{\mathcal{D}},-)}", from=1-4, to=2-7]
	\arrow[""{name=5, anchor=center, inner sep=0}, "{\mathcal{D}'(1_{\mathcal{D}'},-)}"', from=3-4, to=2-7]
	\arrow["\alpha"', shorten <=5pt, shorten >=5pt, Rightarrow, from=2, to=3]
	\arrow["{\nu _G}", curve={height=-6pt}, shorten <=5pt, shorten >=5pt, Rightarrow, from=0, to=1]
	\arrow["{\nu _F}"', curve={height=6pt}, shorten <=5pt, shorten >=5pt, Rightarrow, from=0, to=1]
	\arrow["{!^{\circ }F}"', curve={height=6pt}, shorten <=5pt, shorten >=5pt, Rightarrow, from=4, to=5]
	\arrow["{!^{\circ }G}", curve={height=-6pt}, shorten <=5pt, shorten >=5pt, Rightarrow, from=4, to=5]
\end{tikzcd}\]
which commutes by the commutativity of
\[\begin{tikzcd}
	& {\mathcal{D}'(F1_{\mathcal{D}},Fd)} && {\mathcal{D}'(1_{\mathcal{D}'},Fd)} \\
	{\mathcal{D}(1_{\mathcal{D}},d)} \\
	& {\mathcal{D}'(G1_{\mathcal{D}},Gd)} && {\mathcal{D}'(1_{\mathcal{D}'},Gd)}
	\arrow["F", from=2-1, to=1-2]
	\arrow["G"', from=2-1, to=3-2]
	\arrow["{!^{\circ }}", from=1-2, to=1-4]
	\arrow["{!^{\circ }}"', from=3-2, to=3-4]
	\arrow["{(\alpha _d)_{\circ }}", from=1-4, to=3-4]
\end{tikzcd}\]

        The functor $\mathcal{C}_{-}$ takes a $\kappa$-lex functor $K:\mathcal{C}\to \mathbf{Set}$ to $\mathcal{C}_K=colim _{(x,a)\in (\int K)^{op}} \faktor{\mathcal{C}}{x}$, which is a 2-colimit in the 2-category (i.e.~(2,2)-category) $\mathbf{Cat}$. Following \cite{2filtered} we can also give an explicit description. The objects of $\mathcal{C}_K$ are pairs $(u\to x,a)$ where $u\to x$ is an arrow of $\mathcal{C}$ and $a\in K(x)$. A morphism $[f]:(u\to x,a) \to (v\to y,b)$ is a map which appears between the pullback of $u\to x$ along some $h_1:(r,c)\to (x,a)$ in $\int K$ and the pullback of $v\to y$ along $h_2:(r,c)\to (y,b)$, that is, a commutative triangle

\[\begin{tikzcd}
	& {h^{-1}(x\times v)} && {x\times v} \\
	{h^{-1}(u\times y)} && {u\times y} \\
	& r && {x\times y}
	\arrow[from=1-4, to=3-4]
	\arrow[from=2-3, to=3-4]
	\arrow["{h=\langle h_1,h_2 \rangle}"',"{c\ \xmapsto{K(h)} \ (a,b)}", from=3-2, to=3-4]
	\arrow[from=2-1, to=3-2]
	\arrow[from=1-2, to=3-2]
	\arrow["f", from=2-1, to=1-2]
\end{tikzcd}\]
up to the equivalence of "being identified somewhere", i.e. $f$ (living in $\faktor{\mathcal{C}}{r}$, marked with $c\in K(r)$, the (co)domain moved here through $h:r\to x\times y$) and $g$ (living in $\faktor{\mathcal{C}}{s}$, marked with $d\in K(s)$, the (co)domain moved here through $h':s\to x\times y$) are equivalent if there is a commutative square 

\[\begin{tikzcd}
	r && {x\times y} \\
	t && s
	\arrow["{e\ \mapsto \ d}"', "k", from=2-1, to=2-3]
	\arrow["{h'}", "{d\ \mapsto \ (a,b)}"', from=2-3, to=1-3]
	\arrow["{c\ \mapsto \ (a,b)}", "h"', from=1-1, to=1-3]
	\arrow["{e\ \mapsto \ c}", "{k'}"', from=2-1, to=1-1]
\end{tikzcd}\]
such that taking pullbacks along its edges will identify $f$ and $g$. That is, the (co)domain of the pullback $(k')^{-1}(f)$ can be chosen to equal that of $k^{-1}(g)$, in which case being identified means to become equal. Otherwise there is a unique induced isomorphism between the (co)domains and being identified means that $(k')^{-1}(f)$ and $k^{-1}(g)$ make the square with these isomorphisms commutative.

We remark that (unlike in the usual description of 1-categorical filtered colimits) the object $(u\to x,a)$ is not identified with the object $(h_1^{-1}(u)\to r,c)$, but they are still isomorphic via $[id_{h_1^{-1}(u)}]$.

Now we check that if $\mathcal{C}$ is $\kappa $-lex then so is $\mathcal{C}_K$. The general idea is that given a $\kappa $-small diagram in $\mathcal{C}_K$, we can always chase it to a single slice and compute the limit there. Then, if we add something to this diagram, like a second cone, then this whole picture can be chased to a different slice, while moving between slices preserves the limit cone, so we can get an induced map in this second slice. If we find two induced maps, then again, the two cones with the two connecting maps live in a single slice, so the two maps are equal.

We write out the details in the case of $\kappa $-small products. Let $(u_i\to x_i,a_i)_{i<\gamma <\kappa }$ be a collection of objects. Choose a cone $(h_i:(y,b)\to (x_i,a_i))_i$ in $\int K$ (e.g.~$y=\prod x_i$, $b=\langle a_i\rangle _i$). Then the original collection is isomorphic to $(h_i^{-1}(u_i)\to y,b)_i$. This has a product in $\faktor{\mathcal{C}}{y}$, call it $\widetilde{u}\to y$ with projection maps $p_i:\widetilde{u}\to h_i^{-1}(u_i)$. We claim that $([p_i]:(\widetilde{u}\to y,b)\to 
(h_i^{-1}(u_i)\to y,b))$ is a product cone in $\mathcal{C}_K$. Given an object $(w\to z,c)$ and a cone $([f_i]:(w\to z,c)\to (h_i^{-1}(u_i)\to y,b))_i$ each $[f_i]$ has a representative living in some $\faktor{\mathcal{C}}{z_i}^{c_i}$, as a map between the pullback of $w\to z$ along some $k_i:(z_i, c_i)\to (z,c)$ and the pullback of $h_i^{-1}(u_i)\to y$ along some map $k_i':(z_i,c_i)\to (y,b)$. Choose a cone

\[\begin{tikzcd}
	&& {(\overline{z},\overline{c})} \\
	{(z_i,c_i)} && \dots && {(z_j,c_j)} \\
	& {(y,b)} && {(z,c)}
	\arrow["{l_i}"{description}, dashed, from=1-3, to=2-1]
	\arrow["{l_j}"{description}, dashed, from=1-3, to=2-5]
	\arrow["{k_i'}"{description}, from=2-1, to=3-2]
	\arrow["{k_i}"{description, pos=0.4}, from=2-1, to=3-4]
	\arrow["{k_j'}"{description, pos=0.4}, from=2-5, to=3-2]
	\arrow["{k_j}"{description}, from=2-5, to=3-4]
\end{tikzcd}\]

Write $(k'l)$ instead of $k'_il_i$ as this composite does not depend on $i$, and similarly write $(kl)$ instead of $k_il_i$. As $(k'l)^{-1}$ preserves all limits, we get

\[
\adjustbox{width=\textwidth}{
\begin{tikzcd}
	&&& {l_i^{-1}(k_i^{-1}w)=(kl)^{-1}w} \\
	\\
	{l_i^{-1}((k'_i)^{-1}h_i^{-1}u)=(k'l)^{-1}(h_i^{-1}u_i)} && {(k'l)^{-1}\widetilde{u}} \\
	\\
	&& {\overline{z}}
	\arrow["{l_i^{-1}f_i}"', curve={height=30pt}, from=1-4, to=3-1]
	\arrow["g"{description}, dashed, from=1-4, to=3-3]
	\arrow[color={rgb,255:red,128;green,128;blue,128}, from=1-4, to=5-3]
	\arrow[color={rgb,255:red,128;green,128;blue,128}, from=3-1, to=5-3]
	\arrow["{(k'l)^{-1}p_i}"', from=3-3, to=3-1]
	\arrow[color={rgb,255:red,128;green,128;blue,128}, from=3-3, to=5-3]
\end{tikzcd}
}
\]
and $[g]:(w\to z,c)\to (\widetilde{u}\to y,b)$ is a map which makes the diagram commutative. Uniqueness of such an arrow, as well as the case of equalizers is left to the reader. A more abstract proof can be found in \cite[Theorem 2.9.]{2filtered}.

This argument also shows that each cocone map $\varphi _K^{(x,a)}:\faktor{\mathcal{C}}{x}^a\to \mathcal{C}_K$ is $\kappa $-lex, and given a $\kappa $-lex $\mathcal{D}$ and a functor $F:\mathcal{C}_K\to \mathcal{D}$, $F$ is $\kappa $-lex iff so is each $F\circ \varphi _K^{(x,a)}$.

Let us write out the universal property of $\mathcal{C}_K$ explicitly (cf.~\cite[Theorem 1.18]{2filtered}). Write $Cocone(\mathcal{D})$ for the following category: its objects are pseudococones on $\faktor{\mathcal{C}}{-}$ with top $\mathcal{D}$, and its morphisms are given by modifications. That is, an object consists of a collection of functors $(\psi _{(x,a)}:\faktor{\mathcal{C}}{x}\to \mathcal{D})_{(x,a)\in \int K}$ and for each $f:(x,a)\to (y,b)$ in $\int K$ a natural isomorphism $\theta _f:\psi _{(y,b)}\Rightarrow \psi _{(x,a)}\circ f^{-1}$, such that for each commutative triangle $h=(x,a)\xrightarrow{f}(y,b)\xrightarrow{g}(z,c)$ in $\int K$, the pasting of $\theta _f$, $\theta _g$ and $h^{-1}\Rightarrow f^{-1}\circ g^{-1}$ equals $\theta _h$. A modification is a collection of natural transformations $\rho _{(x,a)}:\psi _{(x,a)}\Rightarrow \psi '_{(x,a)}$ which "fit between the two pseudococones", i.e.~$\theta '_f \circ \rho _{(y,b)}=\rho _{x,a} \circ \theta _f$. 

\[\begin{tikzcd}
	&&&& {\mathcal{D}} \\
	\\
	\\
	{\faktor{\mathcal{C}}{1}*} &&&&&&& {\faktor{\mathcal{C}}{x}a} \\
	&&&& {\faktor{\mathcal{C}}{y}b}
	\arrow[color={rgb,255:red,128;green,128;blue,128}, from=4-1, to=1-5]
	\arrow[curve={height=-24pt}, from=4-1, to=1-5]
	\arrow[""{name=0, anchor=center, inner sep=0}, "{!^{-1}}"{description, pos=0.7}, from=4-1, to=4-8]
	\arrow[""{name=1, anchor=center, inner sep=0}, "{!^{-1}}"{description}, from=4-1, to=5-5]
	\arrow[""{name=2, anchor=center, inner sep=0}, color={rgb,255:red,128;green,128;blue,128}, from=4-8, to=1-5]
	\arrow[""{name=3, anchor=center, inner sep=0}, curve={height=24pt}, from=4-8, to=1-5]
	\arrow[""{name=4, anchor=center, inner sep=0}, color={rgb,255:red,128;green,128;blue,128}, from=5-5, to=1-5]
	\arrow[""{name=5, anchor=center, inner sep=0}, curve={height=-18pt}, from=5-5, to=1-5]
	\arrow["{f^{-1}}"{description}, from=5-5, to=4-8]
	\arrow["\cong"{description}, draw=none, from=0, to=1]
	\arrow["{\rho _{(x,a)}}", shift left=3, shorten <=4pt, shorten >=4pt, Rightarrow, from=3, to=2]
	\arrow["{\theta '_f }"{description}, shift right=5, color={rgb,255:red,128;green,128;blue,128}, Rightarrow, from=4, to=2]
	\arrow["{\theta _f}"{description}, shorten <=7pt, shorten >=9pt, Rightarrow, from=5, to=3]
	\arrow[shift right=2, shorten <=0pt, shorten >=4pt, Rightarrow, from=5, to=4]
\end{tikzcd}\]

Being a 2-colimit means that precomposition with the pseudococone $(\varphi _F^{(x,a)})$ yields an equivalence of categories $\mathbf{Cat}(\mathcal{C}_F,\mathcal{D})\to Cocone(\mathcal{D})$. When $\mathcal{D}$ is $\kappa $-lex, this restricts to an equivalence $\mathbf{Lex}_{\kappa }(\mathcal{C}_K,\mathcal{D})\to Cocone _{lex}(\mathcal{D})$, where $Cocone _{lex}(\mathcal{D})$ is the category of $\kappa $-lex cocones, i.e.~each $\psi _{(x,a)}$ is assumed to be $\kappa $-lex.

We complete the definition of $\mathcal{C}_{-}$ by saying that it sends $\alpha :K\Rightarrow K'$ to 
\[\begin{tikzcd}
	&& {\mathcal{C}_{K'}} \\
	{\mathcal{C}} & \cong \\
	&& {\mathcal{C}_K}
	\arrow["{\varphi _{K'} }", from=2-1, to=1-3]
	\arrow["{\varphi _K =\varphi _K^{(1,*)}}"', from=2-1, to=3-3]
	\arrow["{\mathcal{C}_{\alpha }}"', from=3-3, to=1-3]
\end{tikzcd}\]
where $\mathcal{C}_{\alpha }$ is the induced map between the pseudococones $(\varphi _K^{(x,a)})$ and $(\varphi _{K'}^{(x,\alpha _x(a))})$. In explicit terms it sends the equivalence class $[f]:(u\to x,a)\to (v\to y,b)$ to $[f]:(u\to x, \alpha _x(a))\to (v\to y,\alpha _y(b))$.

It follows that $\Gamma \circ \mathcal{C}_{-}$ is naturally isomorphic to identity, by $\mathcal{C}_K(1,\varphi _K(y))=colim _{(x,a)\in (\int K)^{op}} \faktor{\mathcal{C}}{x}(id_x,\pi _2:y\times x\to x)=colim _{(x,a)\in (\int K)^{op}}\mathcal{C}(x,y)\cong K(y)$. We call the inverse of this isomorphism $\eta _K:K\Rightarrow \Gamma \circ \varphi _K$. It can be described as $K(y)\ni b \mapsto [\Delta _y:(id_y,b)\to (\pi _2:y\times y\to y,b)]\in \mathcal{C}_K(1,\varphi _K(y))$.

Using the equivalence $\delta $ from Theorem \ref{delta}, we get that a pseudococone $(\psi _{(x,a)})$ can be identified with its $(1,*)$-component $M=\psi _{(1,*)}$ equipped with a compatible collection of global elements $(s_{(x,a)}:M(1)\to M(x))_{(x,a)}$. Compatibility follows as $\delta $ is natural in $x$. A modification can be identified with its $(1,*)$-component $\rho =\rho _{(1,*)}:M\Rightarrow N$, which then makes 

\[
\adjustbox{scale=0.9}{
\begin{tikzcd}
	&&& {M(x)^a} \\
	M & {M(1)} &&& {M(y)^b} \\
	&&& {N(x)^a} \\
	N & {N(1)} &&& {N(y)^b}
	\arrow["{M(f)}", from=1-4, to=2-5]
	\arrow["{\rho _x}"{description, pos=0.3}, from=1-4, to=3-4]
	\arrow["\rho"', from=2-1, to=4-1]
	\arrow["{s_{(x,a)}}", from=2-2, to=1-4]
	\arrow["{s_{(y,b)}}"'{pos=0.4}, from=2-2, to=2-5]
	\arrow[from=2-2, to=4-2]
	\arrow["{\rho _y}"{description, pos=0.3}, from=2-5, to=4-5]
	\arrow["{N(f)}", from=3-4, to=4-5]
	\arrow["{s'_{(x,a)}}", from=4-2, to=3-4]
	\arrow["{s'_{(y,b)}}"', from=4-2, to=4-5]
\end{tikzcd}
}
\]
commutative, and conversely, every $M\Rightarrow N$ natural transformation satisfying this property gives rise to a unique modification. 

The category of $\mathcal{C}\to \mathcal{D}$ $\kappa $-lex functors ($M$) equipped with a compatible family of global elements $(s_{x,a}:M(1)\to M(x))_{(x,a)\in \int K}$ is equivalent to the category of $\mathcal{C}\to \mathcal{D}$ $\kappa $-lex functors ($M$) equipped with a natural transformation $colim _{(x,a)\in (\int K)^{op}}\mathcal{C}(x,-) =K\Rightarrow \Gamma M$. Let us write $K\downarrow \Gamma (\mathbf{Lex}_{\kappa }(\mathcal{C},\mathcal{D}))$ for this category.

We have obtained an equivalence (of 1-categories)
\[
\mathbf{Lex}_{\kappa }(\mathcal{C}_K,\mathcal{D})\to K\downarrow \Gamma (\mathbf{Lex}_{\kappa }(\mathcal{C},\mathcal{D}))
\tag{$**$}
\]
which admits the following description: it sends $\widetilde{M}:\mathcal{C}_K\to \mathcal{D}$ to $\widetilde{M}\varphi _K$, equipped with the natural transformation 
\[K(x)\ni a \mapsto \widetilde{M}\left( [\Delta _x]:(id_x,a)\to (\pi _2:x\times x\to x, a) \right): 1\to \widetilde{M}\varphi _K(x) \in \Gamma \widetilde{M}\varphi _K
\]
In other terms this is
\[\begin{tikzcd}
	{\mathcal{C}} && {\mathcal{C}_K} & {\mathcal{D}} & {\mathbf{Set}}
	\arrow["{\varphi _K}"', from=1-1, to=1-3]
	\arrow[""{name=0, anchor=center, inner sep=0}, "K"{description}, curve={height=-70pt}, from=1-1, to=1-5]
	\arrow["{\widetilde{M}}"', from=1-3, to=1-4]
	\arrow[""{name=1, anchor=center, inner sep=0}, "\Gamma"{description}, curve={height=-35pt}, from=1-3, to=1-5]
	\arrow["\Gamma"', from=1-4, to=1-5]
	\arrow["\cong", "{\eta _K}"', shorten <=6pt, shorten >=1pt, Rightarrow, from=0, to=1-3]
	\arrow["{\widetilde{M}}"', shorten <=3pt, Rightarrow, from=1, to=1-4]
\end{tikzcd}\]

For a fixed $\kappa $-lex $M:\mathcal{C}\to \mathcal{D}$, we get an equivalence between the category whose objects are extensions of $M$ and whose morphisms are natural transformations between the extensions
\[\begin{tikzcd}
	&&& {\mathcal{C}_K} \\
	{\mathcal{C}} \\
	&&& {\mathcal{D}}
	\arrow[""{name=0, anchor=center, inner sep=0}, "{\widetilde{M}_1}"{description}, curve={height=18pt}, from=1-4, to=3-4]
	\arrow[""{name=1, anchor=center, inner sep=0}, "{\widetilde{M}_2}"{description}, curve={height=-18pt}, from=1-4, to=3-4]
	\arrow[""{name=2, anchor=center, inner sep=0}, "{\varphi _K}", from=2-1, to=1-4]
	\arrow[""{name=3, anchor=center, inner sep=0}, "M"', from=2-1, to=3-4]
	\arrow["{\alpha }", shorten <=7pt, shorten >=7pt, Rightarrow, from=0, to=1]
	\arrow["{\nu _1}"', curve={height=12pt}, shorten <=5pt, shorten >=5pt, Rightarrow, from=2, to=3]
	\arrow["{\nu _2}", curve={height=-12pt}, shorten <=5pt, shorten >=5pt, Rightarrow, from=2, to=3]
\end{tikzcd}\]
and the category whose objects are given by triples $(M_1:\mathcal{C}\to \mathbf{Set},\  \sigma  :K\Rightarrow \Gamma M_1,\  \nu _1:M_1\xRightarrow{\cong }M)$ and whose morphisms are maps $\alpha ':M_1\Rightarrow M_2$ making 

\[\begin{tikzcd}
	&& {\Gamma M_1} & {M_1} \\
	K & {=} &&& {=} & M \\
	&& {\Gamma M_2} & {M_2}
	\arrow["{\Gamma \alpha '}", from=1-3, to=3-3]
	\arrow["{\nu _1}", from=1-4, to=2-6]
	\arrow["{\alpha '}", from=1-4, to=3-4]
	\arrow["{\sigma _1}", from=2-1, to=1-3]
	\arrow["{\sigma _2}"', from=2-1, to=3-3]
	\arrow["{\nu _2}"', from=3-4, to=2-6]
\end{tikzcd}\]
commutative.

We can restrict this equivalence to get an equivalence between the maximal subgroupoids (i.e.~$\alpha $ and $\alpha '$ are isos). In this case $\alpha '$ is uniquely determined, hence we get the equivalence 
\[
\mathcal{C}\downarrow \mathbf{Lex}^{\sim }_{\kappa }(\mathcal{C}\xrightarrow{\varphi _K} \mathcal{C}_K, \mathcal{C}\xrightarrow{M}\mathcal{D}) \simeq Nat(K,\Gamma M)
\tag{$***$}
\]

Naturality in $K$ and in $M$ is left to the reader.
\end{proof}

\begin{corollary}
\label{Dlocsmall}
    The equivalence $(**)$ (and therefore $(***)$) remains true when $\mathcal{D}$ is only locally small. By taking $\mathcal{D}=\mathbf{Set}$ we obtain an equivalence $\mathbf{Lex}_{\kappa }(\mathcal{C}_K,\mathbf{Set})\to K\downarrow \mathbf{Lex}_{\kappa }(\mathcal{C},\mathbf{Set})$.
\end{corollary}

\begin{definition}
\label{ctosite}
    $\kappa $ is an infinite regular cardinal, $(\mathcal{C},E)$ is a $\kappa $-site. Write $(\mathcal{C},E)\downarrow \mathbf{Site}_{\kappa }^{\sim }$ for the (2,1)-category whose objects are morphisms of $\kappa $-sites $M:(\mathcal{C},E)\to (\mathcal{D},E')$, arrows are triangles
\[\begin{tikzcd}
	&& {(\mathcal{D}_1,E'_1)} \\
	{(\mathcal{C},E)} \\
	&& {(\mathcal{D}_2,E_2')}
	\arrow[""{name=0, anchor=center, inner sep=0}, "{M_1}", from=2-1, to=1-3]
	\arrow["F", from=1-3, to=3-3]
	\arrow[""{name=1, anchor=center, inner sep=0}, "{M_2}"', from=2-1, to=3-3]
	\arrow["{\nu _F}", shorten <=4pt, shorten >=4pt, Rightarrow, from=0, to=1]
\end{tikzcd}\]
where $\nu _F$ is iso, 2-cells are isos fitting between the two triangles.

$\Gamma :(\mathcal{C},E)\downarrow \mathbf{Site}_{\kappa }^{\sim }\to \mathbf{Lex}_{\kappa }(\mathcal{C},\mathbf{Set})$ is defined as in Theorem \ref{addingconst}.
\end{definition}

\begin{corollary}
\label{adjwithsites}
    Let $(\mathcal{C},E)$ and $(\mathcal{D},E')$ be $\kappa $-sites and $M:(\mathcal{C},E)\to (\mathcal{D},E')$ be a morphism of $\kappa $-sites. Then the Hom-groupoids 
    \[
    \adjustbox{width=\textwidth}{
    $
    \mathcal{C}\downarrow \mathbf{Lex}^{\sim }_{\kappa }\left(\mathcal{C}\xrightarrow{\varphi _K} {\mathcal{C}}_{K},\mathcal{C}\xrightarrow{M}\mathcal{D}\right) \text{ and } \mathcal{C}\downarrow \mathbf{Site}^{\sim }_{\kappa }\left((\mathcal{C},E)\xrightarrow{\varphi _K} ({\mathcal{C}}_{K},\varphi _K[E]),(\mathcal{C},E)\xrightarrow{M}(\mathcal{D},E')\right)
    $
    }
    \]
    are equal ($\mathcal{C}\downarrow \mathbf{Lex}_{\kappa }^{\sim }$ is a full subcategory of $(\mathcal{C},E)\downarrow \mathbf{Site}_{\kappa }^{\sim }$), hence the adjunction of Theorem \ref{addingconst} can be seen as
\[\begin{tikzcd}
	{\mathbf{Lex}_{\kappa }(\mathcal{C},\mathbf{Set})} &&& {(\mathcal{C},E)\downarrow \mathbf{Site}^{\sim }_{\kappa }}
	\arrow[""{name=0, anchor=center, inner sep=0}, "{\mathcal{C}_{-}}", curve={height=-12pt}, hook, from=1-1, to=1-4]
	\arrow[""{name=1, anchor=center, inner sep=0}, "\Gamma", curve={height=-12pt}, from=1-4, to=1-1]
	\arrow["\dashv"{anchor=center, rotate=-90}, draw=none, from=0, to=1]
\end{tikzcd}\]
\end{corollary}

\begin{remark}
    This adjunction has a logical meaning: think of $\mathcal{C}$ as a syntactic category. Then $\mathcal{C}_F$ is the syntactic category of an extended theory over an extended signature where we introduced constant symbols along the $\kappa$-lex functor $F:\mathcal{C}\to \mathbf{Set}$. The extended theory consists of the original axioms ($T$) plus formulas like $\varphi (a,y) \Leftrightarrow \exists x(\mu (c,x) \wedge \varphi (x,y))$ whenever $a\in F(x), c\in F(z)$, $\mu (z,x)$ is a $T$-provably functional formula and $F(\mu )$ maps $c$ to $a$.

    By Theorem \ref{addingconst} we have that $\mathcal{C}\to \mathcal{C}_F\xrightarrow{\mathcal{C}_F(1,-)} \mathbf{Set}$ is isomorphic to $F$, so indeed whenever an object $x$ is coming from $\mathcal{C}$, its global elements (constants of sort $x$) are in bijective correspondence with the elements of $F(x)$.

    Note also, that the adjunction says that specifying an extension of $M:\mathcal{C}\to \mathbf{Set}$ to $\mathcal{C}_F$ is the same as finding values for the new constants in a compatible way, i.e.~the same as giving a natural transformation $F\Rightarrow \Gamma M$.
\end{remark}

Now we formulate this adjunction with $\kappa $-toposes in place of $\kappa $-sites.

\begin{definition}
    $\kappa $ is an infinite regular cardinal, $(\mathcal{C},E)$ is a $\kappa $-site. Write $(\mathcal{C},E)\downarrow \mathbf{Topos }_{\kappa }$ for the $(2,1)$-category whose objects are $\kappa $-lex maps $\mathcal{C}\to \mathcal{E}$ which are "$E$-preserving", i.e.~which turn the $E$-families into epimorphic ones, and where $\mathcal{E}$ is a $\kappa $-topos. Arrows (from $M$ to $M'$) are triangles
\[\begin{tikzcd}
	&& {\mathcal{E}} \\
	{\mathcal{C}} \\
	&& {\mathcal{E}'}
	\arrow[""{name=0, anchor=center, inner sep=0}, "M", from=2-1, to=1-3]
	\arrow["{F^*}", from=1-3, to=3-3]
	\arrow[""{name=1, anchor=center, inner sep=0}, "{M'}"', from=2-1, to=3-3]
	\arrow["{\nu _F}", shorten <=4pt, shorten >=4pt, Rightarrow, from=0, to=1]
\end{tikzcd}\]
where $\nu _F$ is any natural transformation, and 2-cells are natural isomorphisms $F^*\Rightarrow G^*$ making the resulting 3-cell commute. 

Let $(\mathcal{C},E)\downarrow \mathbf{Topos}_{\kappa }^{\sim }$ be the subcategory (same objects, less 1-cells, same 2-cells) where $\nu _F$ is assumed to be an isomorphism.
\end{definition}

\begin{remark}
\label{Cdowntopos}
    Write $(\mathcal{C},E)\downarrow \mathbf{Topos }_{\kappa }'$ for the (2,1)-category whose objects are the same, but whose arrows (from $M$ to $M'$) are triangles
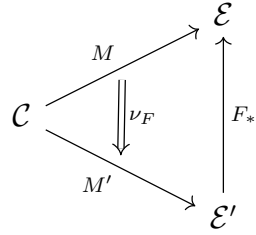
\[\begin{tikzcd}
	&& {\mathcal{E}} \\
	{\mathcal{C}} \\
	&& {\mathcal{E}'}
	\arrow[""{name=0, anchor=center, inner sep=0}, "M", from=2-1, to=1-3]
	\arrow["{F_*}"', from=3-3, to=1-3]
	\arrow[""{name=1, anchor=center, inner sep=0}, "{M'}"', from=2-1, to=3-3]
	\arrow["{\nu _F}", shorten <=4pt, shorten >=4pt, Rightarrow, from=0, to=1]
\end{tikzcd}\]
(with $\nu _F$ arbitrary, 2-cells as before).
    
    There is an isomorphism of (2,1)-categories:
\[\begin{tikzcd}
	{\mathcal{C}\downarrow \mathbf{Topos}_{\kappa }} &&& {\mathcal{C}\downarrow \mathbf{Topos}_{\kappa }'}
	\arrow[""{name=0, anchor=center, inner sep=0}, "{\eta \circ -}", curve={height=-12pt}, from=1-1, to=1-4]
	\arrow[""{name=1, anchor=center, inner sep=0}, "{\varepsilon \circ -}", curve={height=-12pt}, from=1-4, to=1-1]
	\arrow["\simeq"{description}, draw=none, from=0, to=1]
\end{tikzcd}\]
    where $\eta \circ -$ takes a 1-cell to its pasting with the unit $\eta $, and a $2$-cell $\theta $ to its mate
\[\begin{tikzcd}
	&&& {\mathcal{E}} && {\mathcal{E}} && {\mathcal{E}} \\
	{\theta '=} \\
	& {\mathcal{E}'} && {\mathcal{E}'} && {\mathcal{E}'}
	\arrow["{G^*}", from=1-4, to=3-4]
	\arrow["{F^*}", from=1-6, to=3-6]
	\arrow[""{name=0, anchor=center, inner sep=0}, Rightarrow, no head, from=1-4, to=1-6]
	\arrow[""{name=1, anchor=center, inner sep=0}, Rightarrow, no head, from=3-4, to=3-6]
	\arrow[""{name=2, anchor=center, inner sep=0}, Rightarrow, no head, from=3-2, to=3-4]
	\arrow[""{name=3, anchor=center, inner sep=0}, "{G_*}", from=3-2, to=1-4]
	\arrow[""{name=4, anchor=center, inner sep=0}, "{F_*}"', from=3-6, to=1-8]
	\arrow[""{name=5, anchor=center, inner sep=0}, Rightarrow, no head, from=1-6, to=1-8]
	\arrow["\theta"{description}, shorten <=9pt, shorten >=9pt, Rightarrow, from=0, to=1]
	\arrow["{\varepsilon _G}", shorten <=4pt, shorten >=4pt, Rightarrow, from=3, to=2]
	\arrow["{\eta _F}", shorten <=4pt, shorten >=4pt, Rightarrow, from=5, to=4]
\end{tikzcd}\]

Similarly $\varepsilon \circ - $ takes 1-cells to their pasting with the counit and 2-cells to their mates.
\end{remark}

\begin{definition}
    Write $\Gamma : (\mathcal{C},E) \downarrow \mathbf{Topos }_{\kappa }'\to \mathbf{Lex}_{\kappa }(\mathcal{C},\mathbf{Set})$ for the $(2,1)$-functor which takes a 1-cell to its post-composition with global sections:
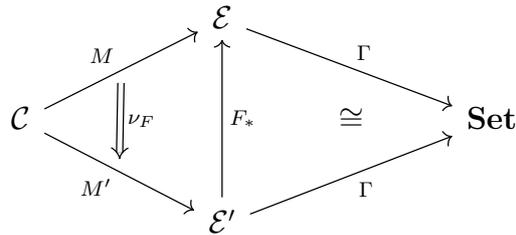
\[\begin{tikzcd}
	&& {\mathcal{E}} \\
	{\mathcal{C}} &&& \cong & {\mathbf{Set}} \\
	&& {\mathcal{E}'}
	\arrow[""{name=0, anchor=center, inner sep=0}, "M", from=2-1, to=1-3]
	\arrow["{F_*}"', from=3-3, to=1-3]
	\arrow[""{name=1, anchor=center, inner sep=0}, "{M'}"', from=2-1, to=3-3]
	\arrow["\Gamma", from=1-3, to=2-5]
	\arrow["\Gamma"', from=3-3, to=2-5]
	\arrow["{\nu _F}", shorten <=4pt, shorten >=4pt, Rightarrow, from=0, to=1]
\end{tikzcd}\]
Here $\cong $ is the unique isomorphism. Since it is unique, $\Gamma $ is a (2,1)-functor: if an isomorphism fits between two triangles then it also makes the other half of the diagram (with the $\Gamma $'s) commute and therefore the two 1-cells are mapped to equal natural transformations.    
\end{definition}

\begin{remark}
We may consider $\mathcal{E}$ as a large $\kappa $-site (with effective epimorphic families as covers), then $\mathcal{E}\simeq Sh(\mathcal{E})$. This way $\eta: 1\Rightarrow F_*F^*$ admits the simple description: if $x\in \mathcal{E}$ then $\eta _x:\mathcal{E}(-,x)\to \mathcal{E'}(F^*(-),F^*(x))$ (in each component) takes an arrow to its $F^*$-image. Therefore the pasting of $\eta $ with the unique isomorphism $\Gamma \circ F_*\Rightarrow \Gamma $ is simply: $F^*: \mathcal{E}(1,x)\to \mathcal{E'}(1,F^*x)$. We will also write $\Gamma $ for $(\mathcal{C},E)\downarrow \mathbf{Topos}_{\kappa }\xrightarrow{\eta \circ -} (\mathcal{C},E)\downarrow \mathbf{Topos}'_{\kappa }\xrightarrow{\Gamma } \mathbf{Lex}_{\kappa }(\mathcal{C},\mathbf{Set})$, as well as for its restriction to $(\mathcal{C},E)\downarrow \mathbf{Topos}^{\sim}_{\kappa }$. It follows that $\Gamma :(\mathcal{C},E)\downarrow \mathbf{Topos}^{\sim}_{\kappa }\to \mathbf{Lex}_{\kappa }(\mathcal{C},\mathbf{Set}) $ coincides with $\Gamma :(\mathcal{C},E)\downarrow \mathbf{Site}^{\sim}_{\kappa }\to \mathbf{Lex}_{\kappa }(\mathcal{C},\mathbf{Set})$ by seeing a topos as a large site (with the canonical topology).
\end{remark}

The following is just the combination of \cite[Theorem 3.9]{presheaftype} and Corollary \ref{adjwithsites}:

\begin{theorem}
    $\kappa $ is an infinite regular cardinal, $(\mathcal{C},E)$ is a $\kappa $-site. Then the functor $\Gamma :(\mathcal{C},E)\downarrow \mathbf{Topos}^{\sim}_{\kappa }\to \mathbf{Lex}_{\kappa }(\mathcal{C},\mathbf{Set})$ admits a (2,1)-categorical left-adjoint which takes $F:\mathcal{C}\to \mathbf{Set}$ to $\mathcal{C}\to \mathcal{C}_F\to Sh(\mathcal{C}_F)$ and a natural transformation $\alpha :F\Rightarrow F'$ to 
\[
\adjustbox{scale=1}{
\begin{tikzcd}
	&& {\mathcal{C}_F} && {Sh(\mathcal{C}_F,\langle \varphi _F[E]\rangle _{\kappa})} \\
	{\mathcal{C}} & \cong && \cong \\
	&& {\mathcal{C}_{F'}} && {Sh(\mathcal{C}_{F'},\langle \varphi _{F'}[E]\rangle _{\kappa})}
	\arrow["{\varphi _F}", from=2-1, to=1-3]
	\arrow["{\mathcal{C}_{\alpha }}", from=1-3, to=3-3]
	\arrow["{\varphi _{F'}}"', from=2-1, to=3-3]
	\arrow["{\mathcal{C}_{\alpha }^*}", from=1-5, to=3-5]
	\arrow["{\#Y}", from=1-3, to=1-5]
	\arrow["{\#Y}"', from=3-3, to=3-5]
\end{tikzcd}
}
\]
\end{theorem}

\begin{remark}
    Unlike in Corollary \ref{adjwithsites}, this adjunction may not prove $\mathbf{Lex}_{\kappa }(\mathcal{C},\mathbf{Set})$ to be a full coreflective subcategory of $(\mathcal{C},E)\downarrow \mathbf{Topos}_{\kappa }^{\sim }$, as the composite 
    \[\mathcal{C}\xrightarrow{\varphi _F} \mathcal{C}_F \xrightarrow{\# Y} Sh(\mathcal{C}_F,\langle \varphi _F[E]\rangle _{\kappa})\xrightarrow{\Gamma } \mathbf{Set}
    \]
    results $x\mapsto Sh(\mathcal{C}_F)(1,\widehat{\varphi _F x})$ instead of $x\mapsto \mathcal{C}_F(1,\varphi _F x)=Fx$. This is not an issue when $(\mathcal{C}_F,\varphi _F[E])$ is subcanonical (assuming $(\mathcal{C},E)$ is), which is the case when $\kappa $ is weakly compact and every $E$-family has $<\kappa $ legs.
\end{remark}

\section{\texorpdfstring{$\mathbf{Set}$}{Set}-valued regular functors are \texorpdfstring{$Sh(B)$}{Sh(B)}-valued models}

This section is a generalisation of Lurie's \cite[Theorem 11., Lecture 16X]{lurie}, which says that $\mathcal{C}\to \mathbf{Set}$ regular functors on a small pretopos $\mathcal{C}$ can be identified with $\mathcal{C}\to Sh(X)$ coherent functors where $X$ is a non-fixed Stone-space. We will prove this in the infinitary setup.

\begin{theorem}
    $\kappa $ is weakly compact. Let $\mathcal{C}$ be $(\kappa ,\kappa )$-coherent with $\kappa $-small disjoint coproducts.
    \begin{itemize}
        \item[i)] If $F:\mathcal{C}\to \mathbf{Set}$ is $\kappa $-lex then $\mathcal{C}_F$ and $\varphi _F:\mathcal{C}\to \mathcal{C}_F$ are $(\kappa ,\kappa )$-coherent, $\mathcal{C}_F$ has $\kappa $-small disjoint coproducts and if $E$ is the set of $\kappa $-small effective epimorphic families in $\mathcal{C}$, then in $\mathcal{C}_F$ the topology $\langle \varphi _F[E]\rangle _{\kappa }$ coincides with the set of $\kappa $-small effective epimorphic families. In particular $Sub_{\mathcal{C}_F}^{\neg }(1)$, the lattice of complemented elements below $1$, is a $(\kappa ,\kappa )$-coherent Boolean-algebra. We write $i:Sub_{\mathcal{C}_F}^{\neg }(1)\to \mathcal{C}_F$ for the $(\kappa ,\kappa )$-coherent inclusion.
        \item[ii)] If $F$ is $\kappa $-regular then $\widetilde{M_F}:\mathcal{C}_F\xrightarrow{Y} Sh(\mathcal{C}_F)\xrightarrow{i_*} Sh(Sub_{\mathcal{C}_F}^{\neg }(1))$ is $(\kappa ,\kappa )$-coherent. 
    \end{itemize}
\end{theorem}

\begin{proof}
    $i)$ The proof that $\mathcal{C}_F$ is $(\kappa ,\kappa )$-coherent with $\kappa $-small disjoint coproducts, and that $\varphi _F$ preserves effective epis and $\kappa $-small disjoint coproducts is analogous to the proof of $\mathcal{C}_F$ being $\kappa $-lex (we use that each $\faktor{\mathcal{C}}{x}$ is $(\kappa ,\kappa )$-coherent with $\kappa $-small disjoint coproducts, and that the pullback functors are $(\kappa ,\kappa )$-coherent). It is left to the reader. The assumption that $\kappa $ is weakly compact is needed, as this way the trees in Definition \ref{lkcoherent} iv) have size $<\kappa $ and they can be chased to a single slice where the transfinite cocompositions of the branches can be proved to form an effective epimorphic family on the root.

    As $\varphi _F$ preserves $\kappa $-small effective epimorphic families, and such families form a $\kappa $-topology on $\mathcal{C}_F$ (again, we use that $\kappa $ is weakly compact, see Example \ref{ktopolsubcanon}), the inclusion $\langle \varphi _F[E]\rangle _{\kappa }\subseteq \{<\kappa \text{ eff.~epic fam.} \}$ follows. For the converse we use that every $\kappa $-small effective epimorphic family is the image of a $\kappa $-small effective epimorphic family along some cocone map $\varphi _F^{(x,a)}: \faktor{\mathcal{C}}{x}\to \mathcal{C}_F$. In $\faktor{\mathcal{C}}{x}$ every effective epic family is the pullback of one coming from $\faktor{\mathcal{C}}{1}$, as
\[\begin{tikzcd}
	{z_i} && {z_i\times x} \\
	y && {y\times x} \\
	& x
	\arrow["{\langle id_{z_i},q_i \rangle}"{description}, from=1-1, to=1-3]
	\arrow["{f_i}"', from=1-1, to=2-1]
	\arrow["{q_i}"{description, pos=0.2}, color={rgb,255:red,128;green,128;blue,128}, from=1-1, to=3-2]
	\arrow["{f_i\times id_x}", from=1-3, to=2-3]
	\arrow["{\pi _2}"{description, pos=0.2}, color={rgb,255:red,128;green,128;blue,128}, from=1-3, to=3-2]
	\arrow["{\langle id_y,p\rangle }"{description}, from=2-1, to=2-3]
	\arrow["p"{description}, color={rgb,255:red,128;green,128;blue,128}, from=2-1, to=3-2]
	\arrow["{\pi _2}"{description}, color={rgb,255:red,128;green,128;blue,128}, from=2-3, to=3-2]
\end{tikzcd}\]
    is a pullback. Since $\varphi _F^{(x,a)}$ preserves pullbacks, we get that in $\mathcal{C}_F$ every $\kappa $-small effective epic family is the pullback of a $\varphi _F[E]$-family, which proves the other inclusion. 
    
    $ii)$ We claim that $1\in \mathcal{C}_F$ is projective, i.e.~every $f:z\twoheadrightarrow 1$ effective epi splits. Such an epi is coming from a slice: it is the equivalence class of $f_0: y\twoheadrightarrow x$ living in  $\faktor{\mathcal{C}}{x}^a$ (that is: in the slice corresponding to $(x,a)\in (\int F)^{op}$). As $F$ is regular, $F(f_0)$ is effective epi, so there is $b\in Fy$ with $F(f_0)(b)=a$. It follows that $(y,b)\to (x,a)$ is an arrow in the indexing diagram and $f_0$ is equivalent to its pullback

\[\begin{tikzcd}
	y \\
	& {y\times _x y} & y \\
	& y & x
	\arrow["{f_0}", from=2-3, to=3-3]
	\arrow["{f_0'}", from=2-2, to=3-2]
	\arrow[from=2-2, to=2-3]
	\arrow[from=3-2, to=3-3]
	\arrow["{id}"{description}, curve={height=12pt}, from=1-1, to=3-2]
	\arrow["{id}"{description}, curve={height=-12pt}, from=1-1, to=2-3]
	\arrow[dashed, from=1-1, to=2-2]
\end{tikzcd}\]
which splits.

It follows, that every $b\hookrightarrow 1$ complemented subobject is projective. Indeed, take an effective epi $f:z\twoheadrightarrow b$. Then $f\sqcup id_{\neg b}$ has a splitting $h$. By forming pullbacks
\[\begin{tikzcd}
	R & {b\sqcup \neg b} \\
	x & {x \sqcup \neg b} \\
	b & {b\sqcup \neg b}
	\arrow[""{name=0, anchor=center, inner sep=0}, hook, from=3-1, to=3-2]
	\arrow["{f~ \sqcup ~id_{\neg b}}", from=2-2, to=3-2]
	\arrow[""{name=1, anchor=center, inner sep=0}, from=2-1, to=2-2]
	\arrow["f"', from=2-1, to=3-1]
	\arrow["h", from=1-2, to=2-2]
	\arrow[""{name=2, anchor=center, inner sep=0}, from=1-1, to=1-2]
	\arrow[from=1-1, to=2-1]
	\arrow["pb"{description}, draw=none, from=1, to=0]
	\arrow["pb"{description}, draw=none, from=2, to=1]
\end{tikzcd}\]
we get that $f$ splits, as the outer rectangle is a pullback along identity.

The composite $\widetilde{M_F}:\mathcal{C}_F\xrightarrow{Y} Sh(\mathcal{C}_F)\xrightarrow{i_*}Sh(Sub_{\mathcal{C}_F}^{\neg }(1))$ is the restricted representable $x\mapsto \restr{\mathcal{C}_F(-,x)}{Sub_{\mathcal{C}_F}^{\neg }(1)}$. It is clear that $\kappa $-small limits and effective epis are preserved. Also $\kappa $-small disjoint coproducts are preserved: given $b\to \bigsqcup x_i$ we can pull back the family $(x_i\to \bigsqcup x_i)_i$ to get a disjoint partition $(b_i\hookrightarrow b)_i$ (note that $b_i\hookrightarrow 1$ is complemented), such that each $b_i\to b\to \bigsqcup x_i$ factors through some $x_j\to \bigsqcup x_i $ (namely we can take $j=i$), so the $\restr{\mathcal{C}_F(-,\bullet )}{Sub_{\mathcal{C}_F}^{\neg }(1)}$-image of $(x_i\to \bigsqcup x_i)_i$ is jointly epimorphic. 
\end{proof}

\begin{proposition}
\label{complemented}
    $\kappa $ is an infinite regular cardinal, $B$ is a $\kappa $-complete Boolean-algebra. Write $\tau _{\kappa -coh}$ for the Grothendieck-topology formed by $\kappa $-small unions. Then $Y:B\hookrightarrow Sh(B,\tau _{\kappa -coh})$ is an isomorphism onto the complemented subobjects of $1\in Sh(B)$. 
\end{proposition}

\begin{proof}
    A subobject is the same as a $\kappa $-complete ideal. Let $I_1$ and $I_2$ be $\kappa $-complete ideals whose intersection is $\{\bot \}$ such that the closure of their union under $ \kappa $-small joins contains $\top $. Since $I_1$ and $I_2$ are closed under $\kappa $-small joins there is $b_1\in I_1$ and $b_2\in I_2$ with $b_1\vee b_2=\top $ (and $b_1 \wedge b_2=\bot $ since it is an element of the intersection). It follows that $I_1=\downarrow b_1$ and $I_2=\downarrow b_2$.
\end{proof}

\begin{proposition}
\label{psi}
    $\kappa $ is weakly compact. Let $\mathcal{C}$ be $(\kappa ,\kappa )$-coherent with $\kappa $-small disjoint coproducts, $B$ a $(\kappa ,\kappa )$-coherent Boolean-algebra (covers are $\kappa $-small unions) and $N:\mathcal{C}\to Sh(B)$ a $(\kappa ,\kappa )$-coherent functor. 

    Then there is an isomorphism $\psi  :B\to Sub_{\mathcal{C}_{\Gamma N}}^{\neg }(1)=colim _{(x,a)\in (\int \Gamma N)^{op}}Sub_{\mathcal{C}}^{\neg }(x)$ given by $b\mapsto [1\sqcup \emptyset \hookrightarrow 1\sqcup 1^{(\restr{*}{b},\restr{*}{\neg b})}]$ which is the equivalence class of $1\sqcup \emptyset \hookrightarrow 1\sqcup 1$ living in $Sub^{\neg }_{\mathcal{C}}(1\sqcup 1)$ marked with  $(\restr{*}{b},\restr{*}{\neg b})\in N(1\sqcup 1)(\top )$. This is the section glued from $N(1)(b)=*\xrightarrow{(i_1)_b}N(1\sqcup 1)(b) $ over $b$ and from $N(1)(\neg b)=*\xrightarrow{(i_2)_{\neg b}}N(1\sqcup 1)(\neg b) $ over its complement. In other terms, it is the map $\widehat{b}\sqcup \widehat{\neg b}\xrightarrow{!\sqcup !} \widehat{\top } \sqcup \widehat{\top }$ in $Sh(B)(\widehat{\top },\widehat{\top }\sqcup \widehat{\top })=Sh(B)(1,1\sqcup 1)=\Gamma 2=\Gamma N2$.

    The inverse of $\psi $ is the following: given $[u\hookrightarrow x^s]$ where $u$ is a complemented subobject of $x$ with complement $u^c$, we map it to $b\in B$, satisfying $\restr{s}{b}\in Nu(b)$ and $\restr{s}{\neg b}\in N(u^c)(\neg b)$.
\end{proposition}

\begin{proof}
    $\psi $ is a homomorphism: $[1\sqcup \emptyset \hookrightarrow 1\sqcup 1 ^{(\restr{*}{\top },-)}]=[1\hookrightarrow 1 ^{*}]$ as the latter subobject is the pullback of the first one via $1\sqcup \emptyset \hookrightarrow 1\sqcup 1$ whose $\Gamma N$-image takes $*$ to $(\restr{*}{\top },-)$. Similarly, $\bot $ is preserved. Given $b,b'$ we have 
    \[
    [1\sqcup \emptyset \hookrightarrow 1\sqcup 1^{(\restr{*}{b},\restr{*}{\neg b})}]=[1\sqcup \emptyset \sqcup 1 \sqcup \emptyset \hookrightarrow 1\sqcup 1\sqcup 1\sqcup 1 ^{(\restr{*}{b\wedge b'},\restr{*}{\neg b\wedge b'},\restr{*}{b\wedge {\neg b'}},\restr{*}{\neg b\wedge {\neg b'}})}]
    \]
    and
    \[
    [1\sqcup \emptyset \hookrightarrow 1\sqcup 1^{(\restr{*}{b'},\restr{*}{{\neg b'}})}]=[1 \sqcup 1 \sqcup \emptyset \sqcup \emptyset \hookrightarrow 1\sqcup 1\sqcup 1\sqcup 1 ^{(\restr{*}{b\wedge b'},\restr{*}{\neg b\wedge b'},\restr{*}{b\wedge {\neg b'}},\restr{*}{\neg b\wedge {\neg b'}})}]
    \]
    Therefore
    \begin{multline*}
         \psi (b)\wedge \psi (b')=[1 \sqcup \emptyset  \sqcup \emptyset \sqcup \emptyset \hookrightarrow 1\sqcup 1\sqcup 1\sqcup 1 ^{(\restr{*}{b\wedge b'},\restr{*}{\neg b\wedge b'},\restr{*}{b\wedge {\neg b'}},\restr{*}{\neg b\wedge {\neg b'}})}]=\\ =[1\sqcup \emptyset \hookrightarrow 1\sqcup 1 ^{(\restr{*}{b\wedge b'},\restr{*}{\neg (b\wedge b')})}]=\psi (b\wedge b')
    \end{multline*}
    and similarly $\psi (b)\vee \psi (b')=\psi (b\vee b')$.

    The proposed inverse is well-defined: first of all, such a $b$ exists. Since $N$ preserves disjoint coproducts, there is a cover $(b_i\to \top )_{i<\gamma <\kappa }$ s.t.~for each $i$ either $\restr{s}{b_i}\in Nu(b_i)$ or $\restr{s}{b_i}\in Nu^c(b_i)$, then define $b$ as the union of those $b_i$'s which belong to the former set. Secondly, it does not depend on the representatives: given $x\xrightarrow{f}y\hookleftarrow v$ and $s\in Nx(\top)$ we get $\restr{Nf_{\top}(s)}{b}=Nf_b(\restr{s}{b})\in Nv(b)$ iff $\restr{s}{b}\in N(f^{-1}v)(b)$ as $N$ preserves pullbacks (and pullbacks are pointwise in $Sh(B)$).

    The composite $B\to Sub_{\mathcal{C}_{\Gamma N}}^{\neg }(1) \to B$ is clearly identity. To check the other composite take $[u\hookrightarrow x^s]$ where $u$ is complemented and $\restr{s}{b}\in Nu(b)$, $\restr{s}{\neg b}\in Nu^c(\neg b)$. Then $x=u\sqcup u^c\xrightarrow{!\sqcup !} 1\sqcup 1$ is such that $u\hookrightarrow x$ is the pullback of $1\sqcup \emptyset $ along it and its $\Gamma N$-image takes $s$ to $(\restr{*}{b},\restr{*}{\neg b})$. So $[u\hookrightarrow x^s]=[1\sqcup \emptyset \hookrightarrow 1\sqcup 1 ^{(\restr{*}{b},\restr{*}{\neg b})}]$.
\end{proof}

\begin{definition}
    $\kappa $ is weakly compact, $\mathcal{C}$ is $(\kappa ,\kappa )$-coherent. We write $\mathcal{C}\downarrow Sh(\mathbf{BAlg}_{\kappa ,\kappa })$ for the category whose objects are $(\kappa ,\kappa )$-coherent functors  $M:\mathcal{C}\to Sh(B)$, where $B$ is a $(\kappa ,\kappa )$-coherent Boolean-algebra (covers are $\kappa $-small unions). 1-cells (from $M$ to $M'$) are triangles 
\[
\adjustbox{scale=0.9}{
\begin{tikzcd}
	&& {Sh(B)} \\
	{\mathcal{C}} \\
	&& {Sh(B')}
	\arrow[""{name=0, anchor=center, inner sep=0}, "M", from=2-1, to=1-3]
	\arrow[""{name=1, anchor=center, inner sep=0}, "{M'}"', from=2-1, to=3-3]
	\arrow["{H_*}"', from=3-3, to=1-3]
	\arrow["{\nu }", shorten <=4pt, shorten >=4pt, Rightarrow, from=0, to=1]
\end{tikzcd}
}
\]
where $H:B\to B'$ is a homomorphism of Boolean-algebras preserving $<\kappa $ meets, $H_*$ is precomposition with $H^{op}$ and $\nu $ is an arbitrary natural transformation.
\end{definition}

\begin{remark}
    The 1-category $\mathcal{C}\downarrow Sh(\mathbf{BAlg}_{\kappa ,\kappa })$ can be seen as a full subcategory in the (2,1)-category $\mathcal{C}\downarrow \mathbf{Topos}_{\kappa }'$ (cf.~Remark \ref{Cdowntopos}). In other terms, 2-cells in this subcategory are trivial.

    Indeed, by Proposition \ref{complemented} the inclusion $Y:B\hookrightarrow Sh(B)$ is an isomorphism from $B$ onto $Sub_{Sh(B)}^{\neg }(1)$. Therefore, given a $\kappa $-lex inverse image map $H^*:Sh(B)\to Sh(B')$ it restricts to a $\kappa $-homomorphism $H_0:B\to B'$, and hence $H^*= H_0^*$ as both preserve colimits. So every $\kappa $-map $Sh(B')\to Sh(B)$ is coming from a $\kappa $-homomorphism $B\to B'$. If $H^*\cong K^*$ (equivalently $H_*\cong K_*$) then $H_0\cong K_0$, consequently $H_0=K_0$, meaning $H_*=K_*$ (precomposition with equal maps is equal).
\end{remark}

\begin{theorem}
\label{Lurie}
      $\kappa $ is weakly compact, $\mathcal{C}$ is $(\kappa ,\kappa )$-coherent with $\kappa $-small disjoint coproducts. We have an equivalence of categories:

\[
\adjustbox{scale=0.9}{
\begin{tikzcd}
	{\mathbf{Reg}_{\kappa }(\mathcal{C},\mathbf{Set})} && {\mathcal{C}\downarrow Sh(\mathbf{BAlg}_{\kappa , \kappa })}
	\arrow[""{name=0, anchor=center, inner sep=0}, "{M_{\bullet }}", curve={height=-12pt}, from=1-1, to=1-3]
	\arrow[""{name=1, anchor=center, inner sep=0}, "{\Gamma }", curve={height=-12pt}, from=1-3, to=1-1]
	\arrow["\simeq"{description}, draw=none, from=0, to=1]
\end{tikzcd}
}
\]

Here $M_{\bullet }$ takes a functor $F$ to $M_F:\mathcal{C}\xrightarrow{\varphi _F}\mathcal{C}_F\xrightarrow{\widetilde{M}_F}Sh(Sub_{\mathcal{C}_F}^{\neg}(1))$ and a natural transformation $\alpha :F\Rightarrow G$ to 

\[
\adjustbox{scale=0.9}{
\begin{tikzcd}
	&&& {Sh(Sub_{\mathcal{C}_F}^{\neg }(1))} \\
	{\mathcal{C}} \\
	&&& {Sh(Sub_{\mathcal{C}_G}^{\neg }(1))}
	\arrow[""{name=0, anchor=center, inner sep=0}, "{M_F}", from=2-1, to=1-4]
	\arrow[""{name=1, anchor=center, inner sep=0}, "{M_G}"', from=2-1, to=3-4]
	\arrow["{(\restr{\mathcal{C}_{\alpha }}{Sub_{\mathcal{C}_F}^{\neg }(1)})_*}"', from=3-4, to=1-4]
	\arrow["{\mu _{\alpha }}", shorten <=4pt, shorten >=4pt, Rightarrow, from=0, to=1]
\end{tikzcd}
}
\]
whose $x$-component 
\begin{multline*}
    \mu _{\alpha ,x}: \restr{\mathcal{C}_F(-,\varphi _F x)}{Sub_{\mathcal{C}_F}^{\neg }(1)}\to \restr{\mathcal{C}_G(\mathcal{C}_{\alpha }(-),\varphi _G x)}{Sub_{\mathcal{C}_F}^{\neg }(1)}\cong \restr{\mathcal{C}_G(\mathcal{C}_{\alpha }(-),\mathcal{C}_{\alpha } \varphi _F x)}{Sub_{\mathcal{C}_F}^{\neg }(1)}
\end{multline*} is simply applying $\mathcal{C}_{\alpha }$ to arrows. In other terms it is

\[
\adjustbox{scale=0.9}{
\begin{tikzcd}
	&& {\mathcal{C}_F} && {Sh(\mathcal{C}_F)} & {Sh(\mathcal{C}_F)} & {Sh(Sub_{\mathcal{C}_F}^{\neg }(1))} \\
	{\mathcal{C}} &&& {\cong } \\
	&& {\mathcal{C}_G} && {Sh(\mathcal{C}_G)} & {Sh(\mathcal{C}_G)} & {Sh(Sub_{\mathcal{C}_G}^{\neg }(1))}
	\arrow[""{name=0, anchor=center, inner sep=0}, "{\varphi _F}", from=2-1, to=1-3]
	\arrow[""{name=1, anchor=center, inner sep=0}, "{\varphi _G}"', from=2-1, to=3-3]
	\arrow["{\mathcal{C}_{\alpha }}", from=1-3, to=3-3]
	\arrow["{Y}", from=1-3, to=1-5]
	\arrow[""{name=2, anchor=center, inner sep=0}, Rightarrow, no head, from=1-5, to=1-6]
	\arrow["{(\mathcal{C}_{\alpha })_*}"', from=3-6, to=1-6]
	\arrow[""{name=3, anchor=center, inner sep=0}, "{(i_F)_*}", from=1-6, to=1-7]
	\arrow[""{name=4, anchor=center, inner sep=0}, "{(i_G)_*}"', from=3-6, to=3-7]
	\arrow["{(\restr{\mathcal{C}_{\alpha }}{Sub_{\mathcal{C}_F}^{\neg }(1)})_*}"', from=3-7, to=1-7]
	\arrow["{(\mathcal{C}_{\alpha })^*}"', from=1-5, to=3-5]
	\arrow[""{name=5, anchor=center, inner sep=0}, Rightarrow, no head, from=3-5, to=3-6]
	\arrow["Y"', from=3-3, to=3-5]
	\arrow["{\nu _{\alpha }}","{\cong }"', shorten <=4pt, shorten >=4pt, Rightarrow, from=0, to=1]
	\arrow["{\eta }"', shorten <=9pt, shorten >=9pt, Rightarrow, from=2, to=5]
	\arrow["{=}", draw=none, from=3, to=4]
\end{tikzcd}
}
\]

$\Gamma $ is post-composition with global sections, as before.
\end{theorem}

\begin{proof}
    If $N:\mathcal{C}\to Sh(B)$ is $(\kappa ,\kappa )$-coherent, where $B$ is a $(\kappa ,\kappa )$-coherent Boolean-algebra, then $\Gamma N$ is $\kappa $-regular, as $\Gamma $ is $\kappa $-regular: it is $\kappa $-lex, we need that it preserves effective epis. This is \cite[Proposition 9.(2),Lecture 16X]{lurie}. We repeat the argument: if $\alpha :F\Rightarrow G$ is an epimorphism of sheaves then for any $s\in G(\top )$ there is a $<\kappa $ covering family $(u_i\to \top )_{i<\gamma <\kappa }$ such that for each $i$: $\restr{s}{u_i}$ has a preimage in $F(u_i)$. But $(u_i)_i$ has a disjoint refinement of the same cardinality with the same property. Then the lifts are automatically compatible, hence they glue together to a lift of $s$ in $F(\top )$. It follows that both $M_{\bullet }$ and $\Gamma $ are well-defined functors. 
    
    $\Gamma \circ M_{\bullet }$ is naturally isomorphic to identity as $\mathcal{C}\xrightarrow{\varphi _F}\mathcal{C}_F\xrightarrow{\Gamma }\mathbf{Set}$ is isomorphic to $F$ (naturally wrt.~$\alpha $, see Theorem \ref{addingconst}) and the topology on $\mathcal{C}_F$ is subcanonical.

    It remains to check that $M_{\bullet }\circ \Gamma $ is naturally isomorphic to identity. Let $\psi =\psi _N  :B\to Sub_{\mathcal{C}_{\Gamma N}}^{\neg }(1)$ be the isomorphism from Proposition \ref{psi}. Note that $\psi _*: Sh(Sub_{\mathcal{C}_{\Gamma N}}^{\neg }(1))\to Sh(B)$ is an isomorphism of categories.
    
    It suffices to find natural isomorphisms $\theta _N: N\Rightarrow \psi _* \circ M_{\Gamma N}$ fitting into the diagram: 

\[\begin{tikzcd}
	&&&&&& {Sh(B)} \\
	&&&&&&&& {Sh(B')} \\
	{\mathcal{C}} &&&&&& {Sh(Sub_{\mathcal{C}_{\Gamma N}}^{\neg }(1))} \\
	&&&&&&&& {Sh(Sub_{\mathcal{C}_{\Gamma N'}}^{\neg }(1))}
	\arrow["{H_*}"{description}, from=2-9, to=1-7]
	\arrow[""{name=2, anchor=center, inner sep=0}, "N"{description, pos=0.4}, curve={height=-12pt}, from=3-1, to=1-7]
	\arrow[""{name=3, anchor=center, inner sep=0}, "{N'}"{description, pos=0.4}, curve={height=-6pt}, from=3-1, to=2-9]
	\arrow[""{name=4, anchor=center, inner sep=0}, "{M_{\Gamma N}}"{description, pos=0.4}, from=3-1, to=3-7]
	\arrow[""{name=5, anchor=center, inner sep=0}, "{M_{\Gamma N'}}"{description, pos=0.4}, curve={height=18pt}, from=3-1, to=4-9]
	\arrow[""{name=0, anchor=center, inner sep=0}, "{\psi _*}"{description, pos=0.6}, from=3-7, to=1-7]
	\arrow[""{name=1, anchor=center, inner sep=0}, "{\psi _*}"{description}, from=4-9, to=2-9]
	\arrow["{(\restr{\mathcal{C}_{\Gamma \alpha }}{ })_*}"{description}, from=4-9, to=3-7]
	\arrow["{=}"{description}, draw=none, from=1, to=0]
        \arrow["{\alpha }"{pos=0.2}, shift right=4, shorten <=0pt, shorten >=35pt, Rightarrow, from=2, to=3]
        \arrow["{\mu _{\alpha }}"{pos=0.2}, shift right=4, shorten <=0pt, shorten >=50pt, Rightarrow, from=4, to=5]
        \arrow["{\theta _{N }}"{description, pos=0.7}, shift right=5, shorten <=20pt, shorten >=0pt, Rightarrow, from=2, to=4]
        \arrow["{\theta _{N '}}"{description}, shift right=11, shorten <=11pt, shorten >=9pt, Rightarrow, from=3, to=5]
\end{tikzcd}
\]

First let us check that the back square indeed commutes, i.e.~that given a 1-cell $(H_*,\alpha )$ we have $\psi _{N'}\circ H =\restr{\mathcal{C}_{\Gamma \alpha }}{} \circ \psi _N$. Fix $b\in B$. The subobject of $1_{\mathcal{C}_{\Gamma N}}$ represented as $[1\sqcup \emptyset \hookrightarrow 1\sqcup 1^{(\restr{*}{b},\restr{*}{\neg b})}]$ is taken to $[1\sqcup \emptyset \hookrightarrow 1\sqcup 1^{\Gamma \alpha ((\restr{*}{b},\restr{*}{\neg b}))}]$ by (the restriction of) $\mathcal{C}_{\Gamma \alpha }$, so we are left to prove $\Gamma \alpha ((\restr{*}{b},\restr{*}{\neg b}))=(\restr{*}{H(b)},\restr{*}{\neg H(b)})\in \Gamma N2=\Gamma 2$.

$\Gamma \alpha $ can be written as
\[
\adjustbox{scale=0.9}{
\begin{tikzcd}
	&& {Sh(B)} & {Sh(B)} \\
	{\mathcal{C}} &&&& {=} & {\mathbf{Set}} \\
	&& {Sh(B')} & {Sh(B')}
	\arrow[""{name=0, anchor=center, inner sep=0}, Rightarrow, no head, from=1-3, to=1-4]
	\arrow["{H^*}"', from=1-3, to=3-3]
	\arrow["\Gamma", from=1-4, to=2-6]
	\arrow[""{name=1, anchor=center, inner sep=0}, "N", from=2-1, to=1-3]
	\arrow[""{name=2, anchor=center, inner sep=0}, "{N'}"', from=2-1, to=3-3]
	\arrow[""{name=3, anchor=center, inner sep=0}, Rightarrow, no head, from=3-3, to=3-4]
	\arrow["{H_*}"', from=3-4, to=1-4]
	\arrow["\Gamma"', from=3-4, to=2-6]
	\arrow["\eta"', shorten <=9pt, shorten >=9pt, Rightarrow, from=0, to=3]
	\arrow["{\varepsilon \alpha}", shorten <=4pt, shorten >=4pt, Rightarrow, from=1, to=2]
\end{tikzcd}
}
\]

Therefore its component at $x$ is given by
\[
Sh(B)(1,Nx)\xrightarrow{H^*} Sh(B')(1,H^*Nx)\xrightarrow{(\varepsilon \alpha )_x\circ -}Sh(B')(1,N'x)
\]
When $x=2$ the second map is an isomorphism (post-composing the comparison map between the two choices of $1\sqcup 1$), so $\Gamma \alpha _2$ takes $\widehat{b}\sqcup \widehat{\neg b}\xrightarrow{!\sqcup !} \widehat{\top }\sqcup \widehat{\top }$ to its $H^*$-image, which is $\widehat{H(b)}\sqcup \widehat{\neg H(b)}\xrightarrow{!\sqcup !} \widehat{\top }\sqcup \widehat{\top }$ and this is what we wanted to prove.

Now we should define $\theta _N:N\Rightarrow \psi _*\circ M_{\Gamma N}$. Fix $x\in \mathcal{C}$ and $b\in B$. Then the components should be maps
\begin{multline*}
\theta _{N,x,b}:Nx(b)\to (\psi _*\circ M_{\Gamma N})(x)(b)=\mathcal{C}_{\Gamma N}([1\sqcup \emptyset \to 1\sqcup 1 ^{(\restr{*}{b},\restr{*}{\neg b})}],\varphi _{\Gamma N}(x) )=\\ =\mathcal{C}_{\Gamma N}([1\sqcup \emptyset \to 1\sqcup 1 ^{(\restr{*}{b},\restr{*}{\neg b})}],[x\sqcup x \to 1\sqcup 1 ^{(\restr{*}{b},\restr{*}{\neg b})}])
\end{multline*}
(so these are representatives living in $\faktor{\mathcal{C}}{1\sqcup 1}$, marked with $(\restr{*}{b},\restr{*}{\neg b})\in \Gamma N(2)$). The Hom-sets in $\mathcal{C}_{\Gamma N}$ are computed as the ($\kappa $-filtered) colimit of Hom-sets in the slices. The colimit is indexed with pairs $(h:z\to 1\sqcup 1,s\in \Gamma Nz)$ where $s$ is taken to $(\restr{*}{b},\restr{*}{\neg b})$ by $\Gamma N(h)$. Such a map is of the form $z_1\sqcup z_2\xrightarrow{!\sqcup !}1\sqcup 1$ and such a global section is the same as a pair of maps $(s_1:\widehat{b}\to Nz_1, s_2:\widehat{\neg b}\to Nz_2)$.

Consequently, the colimit is indexed over $(\int ev_b \circ N)^{op}\times (\int ev_{\neg b} \circ N)^{op}$. At such a pair $((z_1,s_1),(z_2,s_2))$ the corresponding Hom-set (in $\faktor{\mathcal{C}}{z_1\sqcup z_2}$) is between the pullbacks: $z_1\sqcup \emptyset $ and $(z_1\sqcup z_2)\times x$ (sitting over $z_1\sqcup z_2$ via inclusion resp.~projection). This coincides with the Hom-set $\mathcal{C}(z_1,x)$.

We have arrived to the conclusion 
\begin{multline*}
    \psi _* \circ M_{\Gamma N}(x)(b)\cong colim _{((z_1,s_1),(z_2,s_2))\in(\int ev_b \circ N)^{op}\times (\int ev_{\neg b} \circ N)^{op}}\mathcal{C}(z_1,x)  \\ \cong colim _{(z,s)\in (\int ev_b \circ N)^{op}}\mathcal{C}(z,x)
\end{multline*}
which isomorphism is natural in $x$ and in $b$ and we define $\theta _{N,\bullet ,b}$ as the inverse of the well-known isomorphism (writing $ev_b\circ N$ as a colimit of representables) whose $x$-component $colim _{(z,s)\in (\int ev_b \circ N)^{op}}\mathcal{C}(z,x) \to (ev_b \circ N)(x)=Nx(b)$ is sending an element $f:z\to x$ of the $s:\widehat{b}\to Nz$-indexed Hom-set to the composite $\widehat{b}\xrightarrow{s}Nz\xrightarrow{Nf}Nx$.

It remains to check naturality in $\alpha $, i.e.~to check that the above 3-cell commutes. We need the commutativity of

\[
\adjustbox{width=\textwidth}{
\begin{tikzcd}
	{\mathcal{C}_{\Gamma N}\left(\left[\overset{1\sqcup \emptyset}{ \underset{ 1\sqcup 1}{\downarrow }}\right]^{(\restr{*}{b},\restr{*}{\neg b})},\varphi _{\Gamma  N}(x)\right)} && {\mathcal{C}_{\Gamma N'}\left(\left[\overset{1\sqcup \emptyset}{ \underset{ 1\sqcup 1}{\downarrow }}\right]^{(\restr{*}{H(b)},\restr{*}{\neg H(b)})},\varphi _{\Gamma  N'}(x)\right)} \\
	\\
	{\underset{s_1\sqcup s_2:\ \widehat{b}\sqcup \widehat{\neg b}\to Nz_1\sqcup Nz_2}{colim}\faktor{\mathcal{C}}{z_1\sqcup z_2}\left(\overset{z_1\sqcup \emptyset}{ \underset{z_1\sqcup z_2}{\downarrow }}, \overset{(z_1\sqcup z_2)\times x}{\underset{ z_1\sqcup z_2}{\downarrow }}\right)} && {\underset{s'_1\sqcup s'_2:\ \widehat{H(b)}\sqcup \widehat{\neg H(b)}\to N'z_1\sqcup N'z_2}{colim}\faktor{\mathcal{C}}{z_1\sqcup z_2}\left(\overset{z_1\sqcup \emptyset}{ \underset{z_1\sqcup z_2}{\downarrow }}, \overset{(z_1\sqcup z_2)\times x}{\underset{ z_1\sqcup z_2}{\downarrow }}\right)} \\
	\\
	{colim _{s_1:\ \widehat{b}\to Nz_1}\mathcal{C}(z_1,  x)} && {colim _{s'_1:\ \widehat{H(b)}\to N'z_1}\mathcal{C}(z_1,  x)} \\
	\\
	{Nx(b)} && {N'x(H(b))}
	\arrow["{(\psi _* (\mu_{\alpha }))_{x,b} =\mu_{\alpha ,x,\psi (b)} =\mathcal{C}_{\Gamma \alpha }(-)}", from=1-1, to=1-3]
	\arrow[Rightarrow, no head, from=1-1, to=3-1]
	\arrow[Rightarrow, no head, from=1-3, to=3-3]
	\arrow["{\langle i_{\Gamma \alpha (s_1\sqcup s_2)}\rangle _{s_1\sqcup s_2}}", from=3-1, to=3-3]
	\arrow[Rightarrow, no head, from=3-1, to=5-1]
	\arrow[Rightarrow, no head, from=3-3, to=5-3]
	\arrow["{\langle i_{\widehat{H(b)}\xrightarrow{H^*s_1} H^*Nz_1\xrightarrow{\varepsilon \alpha }N'z_1 }\rangle _{s_1}}", from=5-1, to=5-3]
	\arrow["{\overline{\theta _{N,x,b}}}"', from=5-1, to=7-1]
	\arrow["{\overline{(H_*(\theta _{N'}))_{x,b}}=\overline{\theta _{N',x,H(b)}}}", from=5-3, to=7-3]
	\arrow["{\alpha _{x,b}}"', from=7-1, to=7-3]
\end{tikzcd}
}
\]
Here the two upper squares come from the previously discussed identifications. The commutativity of the lower square means
\[
\alpha _{x,b}(\widehat{b}\xrightarrow{s_1}Nz_1\xrightarrow{Nf}Nx)\ =\ \widehat{H(b)}\xrightarrow{H^*s_1}H^*Nz_1\xrightarrow{\varepsilon \alpha _{z_1}}N'z_1\xrightarrow{N'f}N'x
\]
This is the case, as $\alpha =(\varepsilon \alpha )\circ \eta N$ (where $\varepsilon \alpha $ stands for $\varepsilon N' \circ \alpha $) and therefore $\widehat{b}\xrightarrow{s_1}Nz_1\xrightarrow{Nf}Nx$ is taken to $\widehat{H(b)}\xrightarrow{H^*s_1}H^*Nz_1\xrightarrow{H^*Nf}H^*Nx\xrightarrow{\varepsilon \alpha _x}N'x$ and this is what we want by the naturality of $\varepsilon \alpha $.

\end{proof}

\section{Applications}

\subsection{\texorpdfstring{$Sh(B)$}{Sh(B)}-valued completeness}

(Our version of) Lurie's theorem says that when $\kappa $ is weakly compact, $\kappa $-regular $\mathbf{Set}$-valued functors on $(\kappa ,\kappa )$-co\-he\-rent categories with $\kappa $-small disjoint coproducts are the same as $Sh(\mathbf{BAlg}_{\kappa ,\kappa })$-models. Makkai's theorems (see below) say that there are many $\kappa $-regular $\mathbf{Set}$-valued functors on a $\kappa $-regular category. These together imply that there are many $Sh(\mathbf{BAlg}_{\kappa ,\kappa })$-models on a $(\kappa ,\kappa )$-coherent category.

\begin{theorem}[{\cite[Theorem 5.1]{barrex}}]
\label{m1}
    Let $\mathcal{C}$ be $\kappa $-regular and Barr-exact. Then the evaluation functor $ev_{\bullet }:\mathcal{C}\to [\mathbf{Reg}_{\kappa }(\mathcal{C},\mathbf{Set}), \mathbf{Set}]$ is $\kappa $-regular, fully faithful and its essential image consists of those functors $ \mathbf{Reg}_{\kappa }(\mathcal{C},\mathbf{Set}) \to \mathbf{Set}$ which preserve $\kappa $-filtered colimits and all products.
\end{theorem}

\begin{theorem}[{\cite[Proposition 6.4]{barrex}}]
\label{m2}
    Let $\mathcal{C},\mathcal{D}$ be $\kappa $-regular, Barr-exact. Given a functor $F:\mathbf{Reg}_{\kappa }(\mathcal{D},\mathbf{Set})\to \mathbf{Reg}_{\kappa }(\mathcal{C},\mathbf{Set})$, it is isomorphic to $-\circ M$ for some $\kappa $-regular $M:\mathcal{C}\to \mathcal{D}$ iff it preserves $\kappa $-filtered colimits and all products.
\end{theorem}

\begin{theorem}[{\cite[Theorem 2.3]{barrex}}]
\label{m3}
    For any $\kappa $-regular $\mathcal{C}$ there is a conservative $\kappa $-regular functor $\mathcal{C}\to \mathbf{Set}^I$.
\end{theorem}

We get that $(\kappa ,\kappa )$-coherent categories are complete wrt.~$Sh(\mathbf{BAlg}_{\kappa ,\kappa })$-models:

\begin{theorem}
\label{completeness}
    $\kappa $ is weakly compact. Let $\mathcal{C}$ be a $(\kappa ,\kappa )$-coherent category. Then given $v\hookrightarrow x\hookleftarrow u$ we have $v\leq u$ iff for any $(\kappa ,\kappa )$-coherent Boolean-algebra $B$ and any $(\kappa ,\kappa )$-coherent functor $M:\mathcal{C}\to Sh(B)$: $Mv\leq Mu$ holds.
\end{theorem}

\begin{proof}
    There is a fully faithful $(\kappa ,\kappa )$-coherent $\mathcal{C}\to \mathcal{C}'$ functor where $\mathcal{C}'$ is $(\kappa ,\kappa )$-coherent with $\kappa $-small disjoint coproducts: simply take such a small subcategory in $Sh(\mathcal{C})$. So we assume that $\mathcal{C}$ has $\kappa $-small disjoint coproducts. Assume $v\cap u \neq v$. By Makkai's theorem there is a $\kappa $-regular functor $F:\mathcal{C}\to \mathbf{Set}$ with $F(v\cap u)\neq F(v)$ (as subobjects of $F(x)$). But then $M_F(v\cap u)\neq M_F(v)$ since $F\cong \Gamma M_F$.
\end{proof}

\begin{remark}
    If $\kappa $ is strongly compact then $L_{\kappa ,\kappa }$ satisfies the compactness theorem, and if $\kappa $ is weakly compact then $L_{\kappa ,\kappa }$ satisfies the weak compactness theorem (cf.~\cite[Exercise 20.1.~and Theorem 17.13.]{jech}). From this one can prove that Theorem \ref{completeness} holds with $B=2$, either if
$\mathcal{C}$ is an arbitrary $(\kappa ,\kappa )$-coherent category and $\kappa $ is strongly compact or if $\mathcal{C}$ is a $(\kappa ,\kappa )$-coherent category satisfying a certain smallness assumption and $\kappa $ is weakly compact (see \cite[Proposition 3.1.5.]{inflog}).
\end{remark}

\begin{theorem}
\label{fromintpret}
   $\kappa $ is weakly compact. Let $\mathcal{C},\mathcal{D}$ be $(\kappa ,\kappa )$-pretoposes (which means: $(\kappa ,\kappa )$-coherent with $\kappa $-small disjoint coproducts and Barr-exact). Given a functor $F:\mathbf{Coh}_{\kappa ,\kappa }(\mathcal{D},\mathbf{Set})\to \mathbf{Coh}_{\kappa ,\kappa }(\mathcal{C},\mathbf{Set})$ it is of the form $-\circ M$ for some $(\kappa ,\kappa )$-coherent $M:\mathcal{C}\to \mathcal{D}$ iff there is a lift
\[\begin{tikzcd}
	{\mathbf{Coh}_{\kappa ,\kappa }(\mathcal{D},\mathbf{Set})} && {\mathbf{Coh}_{\kappa ,\kappa }(\mathcal{C},\mathbf{Set})} \\
	{\mathcal{D}\downarrow Sh(\mathbf{BAlg}_{\kappa ,\kappa })} && {\mathcal{C}\downarrow Sh(\mathbf{BAlg}_{\kappa ,\kappa })}
	\arrow["F", from=1-1, to=1-3]
	\arrow[hook', from=1-1, to=2-1]
	\arrow[hook', from=1-3, to=2-3]
	\arrow["{\widetilde{F}}"', dashed, from=2-1, to=2-3]
\end{tikzcd}\]
    such that $\widetilde{F}$ preserves products and $\kappa $-filtered colimits. (I.e.~if we can define $F$ not only on $\mathbf{Set}$-models but also on $Sh(\mathbf{BAlg}_{\kappa ,\kappa })$-models, moreover this extension has the mentioned property.)
\end{theorem}

\begin{proof}
    By Theorem \ref{Lurie} and by Makkai's Theorem \ref{m2}, our condition is equivalent to $\widetilde{F}$ (and therefore $F$) being of the form $-\circ M$ for some $\kappa $-regular $M$. But if for each $N:\mathcal{D}\to Sh(B)$ $(\kappa ,\kappa )$-coherent functor the composite $NM$ is also $(\kappa ,\kappa )$-coherent, it follows that $M$ preserves $\kappa $-small unions: indeed, if $\bigcup _i M(u_i) \hookrightarrow M(\bigcup _i u_i)$ is proper then by the previous theorem for some $N$ so is $\bigcup _i NM(u_i)= N(\bigcup _i M(u_i)) \hookrightarrow NM(\bigcup _i u_i)$.
\end{proof}

\subsection{Elementary maps}

We relate regular functors to products of coherent ones. The following is in \cite[p.124]{barr}.

\begin{definition}
    $\mathcal{C}$, $\mathcal{D}$ are categories with finite limits and $F,G:\mathcal{C}\to \mathcal{D}$ are lex functors. A natural transformation $\alpha :F\Rightarrow G$ is elementary if the naturality squares at monomorphisms are pullbacks.
\end{definition}

\begin{proposition}
\label{elemthenmonic}
    $\mathcal{C}$, $\mathcal{D}$, $F,G:\mathcal{C}\to \mathcal{D}$ are lex. If $\alpha :F\Rightarrow G$ is elementary then each component $\alpha _x$ is monic.
\end{proposition}

\begin{proof}
    If in 
\[\begin{tikzcd}
	d \\
	& Fx && {F(x\times x)=F(x)\times F(x)} \\
	&& pb \\
	& Gx && {G(x\times x)=G(x)\times G(x)}
	\arrow["h", dashed, from=1-1, to=2-2]
	\arrow["{\langle f,g\rangle}", curve={height=-18pt}, from=1-1, to=2-4]
	\arrow["{\alpha _x g}"{description}, curve={height=12pt}, from=1-1, to=4-2]
	\arrow["{F(\Delta _x)}", hook, from=2-2, to=2-4]
	\arrow["{\alpha _x}"', from=2-2, to=4-2]
	\arrow["{\alpha _{x\times x}=\alpha _x \times \alpha _x}", from=2-4, to=4-4]
	\arrow["{G(\Delta _x)}"', hook, from=4-2, to=4-4]
\end{tikzcd}\]
$\alpha _xf=\alpha _xg$ then the outer square commutes, hence there is $h$ making both triangles commute, therefore $f=h=g$.
\end{proof}

We will prove that for the unit of a geometric morphism the converse also holds.

\begin{definition}
    $\mathcal{C}$, $\mathcal{D}$, $F:\mathcal{C}\to \mathcal{D}$ are lex. We say that $F$ is strongly conservative if for arrows $c\xrightarrow{f}x\xhookleftarrow{i} u$ the existence of a lift on the left side implies the existence of a lift on the right side:
\[\begin{tikzcd}
	&& Fu &&&& u \\
	Fc && Fx && c && x
	\arrow["Ff", from=2-1, to=2-3]
	\arrow["Fi", hook', from=1-3, to=2-3]
	\arrow[dashed, from=2-1, to=1-3]
	\arrow["f", from=2-5, to=2-7]
	\arrow["i", hook', from=1-7, to=2-7]
	\arrow[dashed, from=2-5, to=1-7]
\end{tikzcd}\]
$F$ is conservative if the above property holds when $f$ is a mono (equivalently: if $F:Sub_{\mathcal{C}}(x)\to Sub_{\mathcal{D}}(Fx)$ is injective, equivalently: if $F$ reflects isomorphisms).
\end{definition}

\begin{remark}
\label{regcons}
If $\mathcal{C}$, $\mathcal{D}$ and $F:\mathcal{C}\to \mathcal{D}$ are regular then $F$ is conservative iff it is strongly conservative (the lift $c\to u$ exists iff $im(f)\subseteq u$). 
\end{remark}

\begin{proposition}
\label{strongconsadj}
 $\mathcal{C}$, $\mathcal{D}$ are lex. Let 
\[\begin{tikzcd}
	{\mathcal{C}} &&& {\mathcal{D}}
	\arrow[""{name=0, anchor=center, inner sep=0}, "L", curve={height=-12pt}, from=1-1, to=1-4]
	\arrow[""{name=1, anchor=center, inner sep=0}, "R", curve={height=-12pt}, from=1-4, to=1-1]
	\arrow["\dashv"{anchor=center, rotate=-90}, draw=none, from=0, to=1]
\end{tikzcd}\]
be an adjunction with left-exact left adjoint. Then the unit $\eta $ is elementary iff $L$ is strongly conservative.
\end{proposition}

\begin{proof}
    As the Yoneda-embedding preserves and reflects pullbacks, we need that the upper and hence equivalently the outer square in
\[\begin{tikzcd}
	{\mathcal{C}(c,u)} && {\mathcal{C}(c,x)} \\
	{\mathcal{C}(c,RLu)} && {\mathcal{C}(c,RLx)} \\
	{\mathcal{D}(Lc,Lu)} && {\mathcal{D}(Lc,Lx)}
	\arrow["{i_{\circ}}", hook, from=1-1, to=1-3]
	\arrow["{(\eta _x)_{\circ}}"', from=1-3, to=2-3]
	\arrow["{(\eta _u)_{\circ}}", from=1-1, to=2-1]
	\arrow["{(RLi)_{\circ}}"', hook, from=2-1, to=2-3]
	\arrow["\cong", from=2-1, to=3-1]
	\arrow["{(Li)_{\circ}}"', hook, from=3-1, to=3-3]
	\arrow["\cong"', from=2-3, to=3-3]
	\arrow["L"{pos=0.8}, curve={height=-35pt}, from=1-3, to=3-3]
	\arrow["L"'{pos=0.8}, curve={height=35pt}, from=1-1, to=3-1]
\end{tikzcd}\]
    is a pullback, for every mono $i:u\hookrightarrow x$ and every object $c$ in $\mathcal{C}$. This means exactly that $L$ is strongly conservative.
\end{proof}

\begin{proposition}
    $\mathcal{C}$, $\mathcal{D}$ are regular, in $\mathcal{C}$ every mono is regular. Let 
\[\begin{tikzcd}
	{\mathcal{C}} &&& {\mathcal{D}}
	\arrow[""{name=0, anchor=center, inner sep=0}, "L", curve={height=-12pt}, from=1-1, to=1-4]
	\arrow[""{name=1, anchor=center, inner sep=0}, "R", curve={height=-12pt}, from=1-4, to=1-1]
	\arrow["\dashv"{anchor=center, rotate=-90}, draw=none, from=0, to=1]
\end{tikzcd}\]
be an adjunction with left-exact left adjoint. Then the following are equivalent: 1) $L$ is faithful 2) $L$ is conservative 3) $L$ is strongly conservative 4) $\eta $ is elementary 5) $\eta $ is pointwise mono.
\end{proposition}

\begin{proof}
    \fbox{1) $\Rightarrow $ 2)} if $i:u\hookrightarrow x$ is the equalizer of $f,g:x\to y$ then $Li$ being iso implies $Lf=Lg$ implies $f=g$ implies $i$ being an iso. \fbox{2) $\Rightarrow $ 3)} Remark \ref{regcons}. \fbox{3) $\Rightarrow $ 4)} Proposition \ref{strongconsadj}. \fbox{4) $\Rightarrow $ 5)} Proposition \ref{elemthenmonic} \fbox{5) $\Rightarrow $ 1)} well-known (similar to Proposition \ref{strongconsadj}).
\end{proof}

\begin{corollary}
\label{unitelem}
    If the unit of a geometric morphism is pointwise mono then it is elementary.
\end{corollary}

\begin{theorem}
\label{regtoprodcoh}
    $\kappa =\aleph _0$ or $\kappa $ is strongly compact. Let $\mathcal{C}$ be $(\kappa ,\kappa )$-coherent with $\kappa $-small disjoint coproducts and let $F:\mathcal{C}\to \mathbf{Set}$ be a $\kappa $-regular functor. Then there are $(\kappa ,\kappa )$-coherent functors $M_i:\mathcal{C}\to \mathbf{Set}$ for which an elementary natural transformation $F\Rightarrow \prod _i M_i$ exists.
\end{theorem}

\begin{proof}
  There is a $(\kappa ,\kappa )$-coherent functor $M:\mathcal{C}\to Sh(B,\tau _{\kappa -coh})$ for some $(\kappa ,\kappa )$-coherent Boolean-algebra $B$ with $\Gamma M=F$. By the cardinality assumption there are enough $\kappa $-complete ultrafilters on $B$, that is, we can embed it to a power-set Boolean-algebra with a homomorphism $J:B\hookrightarrow 2^I$ which preserves $\kappa $-small $\bigwedge $ and $\bigvee $ (see \cite[Proposition 3.1.9.]{inflog}). Recall that $\tau_{\kappa -coh}$ denotes the topology formed by $\kappa $-small unions. Let $\tau _{can}$ be the topology formed by arbitrary unions. We consider $J$ as a site-morphism $(B,\tau _{\kappa -coh})\to (2^{I},\tau _{can})$ and get:

\[
\adjustbox{width=\textwidth}{
\begin{tikzcd}
	&& {Sh(B,\tau _{\kappa -coh})} && {Sh(B,\tau _{\kappa -coh})} \\
	{\mathcal{C}} & {=} &&&&& {\mathbf{Set}} \\
	&& {Sh(2^I,\tau _{can})=\mathbf{Set}^I} && {Sh(2^I,\tau _{can})=\mathbf{Set}^I}
	\arrow["M", from=2-1, to=1-3]
	\arrow["{J^*}", from=1-3, to=3-3]
	\arrow["{J^*M=\langle M_i\rangle _i}"', from=2-1, to=3-3]
	\arrow[""{name=0, anchor=center, inner sep=0}, Rightarrow, no head, from=1-3, to=1-5]
	\arrow[""{name=1, anchor=center, inner sep=0}, Rightarrow, no head, from=3-3, to=3-5]
	\arrow["{J_*}"', from=3-5, to=1-5]
	\arrow[""{name=2, anchor=center, inner sep=0}, "\Gamma", from=1-5, to=2-7]
	\arrow[""{name=3, anchor=center, inner sep=0}, "\Gamma"', from=3-5, to=2-7]
	\arrow["\eta"', shorten <=9pt, shorten >=9pt, Rightarrow, from=0, to=1]
	\arrow["\cong"', draw=none, from=2, to=3]
\end{tikzcd}
}
\]
We need that $J^*$ is conservative, then by Corollary \ref{unitelem} it follows that $\eta $ is elementary. That is, we have to check that if a family $(b_i\leq b)_{i<\mu }$ is sent to a cover by $J$, then it has a $\kappa $-small subfamily whose union is $b$. In other terms: if for any $\kappa $-complete ultrafilter $U$ with $b\in U$ there is some $i$ s.t.~$b_i\in U$ then there is a $\kappa $-small subfamily whose union is $b$. But otherwise the $\kappa $-small unions formed among the $b_i$'s would form a $\kappa $-complete ideal which is disjoint from $\uparrow b$, and as $B$ is $(\kappa ,\kappa )$-coherent and $\kappa $ is strongly compact we could divide these sets by a $\kappa $-complete ultrafilter, contradicting the assumption. 

To see $Sh(2^I, \tau _{can})=\mathbf{Set}^I$ note that $2^I$ is the lattice of open sets in the discrete topological space $I$ and $Sh(I)=\mathbf{Set}^I$ is obvious. Finally, $\Gamma \circ M=F$, $\Gamma \circ \langle M_i \rangle _i =\prod _i M_i$.
\end{proof}

As often, $\kappa $ can be assumed to be weakly compact at the price of a cardinality restriction on $\mathcal{C}$ and on $F$.

\begin{theorem}
    $\kappa $ is weakly compact. Let $\mathcal{C}$ be a $(\kappa ,\kappa )$-coherent category with $\kappa $-small disjoint coproducts and $F:\mathcal{C}\to \mathbf{Set}$ be a $\kappa $-regular functor. Assume that $|\mathcal{C}|\leq \kappa $ and that $|Fx|\leq \kappa $ for each $x\in \mathcal{C}$. Then there are $(\kappa ,\kappa )$-coherent functors $M_i:\mathcal{C}\to \mathbf{Set}$ for which an elementary natural transformation $F\Rightarrow \prod _i M_i$ exists.
\end{theorem}

\begin{proof}
    Following the argument of the previous proof it is enough to prove that $Sub_{\mathcal{C}_F}^{\neg }(1)$ has enough $\kappa $-complete ultrafilters. However since $|\int F|\leq \kappa $ it follows that $|\mathcal{C}_F|\leq \kappa $ and hence $|Sub_{\mathcal{C}_F}^{\neg }(1)|\leq \kappa $. We refer to \cite[Proposition 3.1.9.]{inflog} again.
\end{proof}

\begin{question}
Can we eliminate the need for disjoint coproducts? E.g.~when $\mathcal{C}$ is a distributive lattice $(\kappa =\aleph _0)$ these statements simplify to "every filter is the intersection of prime filters" which is true.
\end{question}

\begin{question}
In Theorem \ref{addingconst} $\Gamma :(\mathcal{C},E)\downarrow \mathbf{Topos }_{\kappa }^{\sim }\to \mathbf{Lex}_{\kappa }(\mathcal{C},\mathbf{Set})$ was proved to have a (2,1)-categorical left adjoint $\mathcal{C}_{-}$. Is it true that $\Gamma :(\mathcal{C},E)\downarrow \mathbf{Topos }_{\kappa }\to \mathbf{Lex}_{\kappa }(\mathcal{C},\mathbf{Set})$ has a (2,1)-categorical left adjoint $Spec$ and that the fixed points of this adjunction result the equivalence in Theorem \ref{Lurie} when $\mathcal{C}$ is $(\kappa ,\kappa )$-coherent ($\kappa $ is weakly compact) and $E=\ <\kappa $ effective epic families?

In fact such a left adjoint was constructed in \cite{spec} when $\mathcal{C}$ is an $\omega $-site (plus it satisfies some additional requirements, e.g.~the legs of the covers are monic).
\end{question}

\begin{question}
    We only used that $\kappa $ is weakly compact to make sure that the topology on $\mathcal{C}_F$ is subcanonical (so that e.g.~$Sub_{\mathcal{C}_F}^{\neg }(1)$ is a $(\kappa ,\kappa )$-coherent Boolean-algebra). Is there a version without this assumption, where one is working with $Sh(\mathcal{C}_F)$ instead of $\mathcal{C}_F$ itself?
\end{question}

\subsubsection*{Declarations}

\textbf{Conflict of interest} The author declares that there are no Conflict of interest.

\subsubsection*{Funding}

Supported by the Grant Agency of the Czech Republic under the grant
22-02964S.

\printbibliography

\end{document}